\documentclass[intlimits,a4paper,10pt,reqno]{amsproc}
\usepackage{amsmath,color}
\usepackage{amssymb}
\usepackage{amsfonts}
\usepackage{amsxtra}
\usepackage{graphicx}
\usepackage{xspace}
\usepackage{multirow}
\usepackage{enumitem}
\usepackage[active]{srcltx}

\definecolor{rltred}{rgb}{0.75,0,0}
\definecolor{rltgreen}{rgb}{0,0.5,0}
\definecolor{rltblue}{rgb}{0,0,0.75}

\usepackage[%
   colorlinks=true,
   urlcolor=rltblue,       % \href{...}{...} external (URL)
   filecolor=rltgreen,     % \href{...} local file
   citecolor=rltgreen,     % \cite{...} 
   linkcolor=rltred,       % \ref{...} and \pageref{...}
   bookmarks=true,
   unicode,
   plainpages=false,
   ]{hyperref}

\usepackage[operator,
rose,frak]
{paper_diening}
\usepackage{amsmath}
\usepackage{amssymb,amsthm,amsxtra,amscd}

\newcommand{\eps}{\epsilon}%

\newcommand{\td}{\partial_{\tau}}

\renewcommand{\O}{\Omega}

\newcommand{\eR}{{\mathbb R}}

\newcommand{\T}{\mathbf{S}}
\newcommand{\F}{\mathbf{F}}

\DeclareMathOperator{\spt}{spt}

\newcommand{\intO}{\int_\O}

\newcommand{\para}{{\delta}}
\newcommand{\letter}{\kappa}

\newcommand{\dist}{\operatorname{dist}}

\newcommand{\bue}{\bu_{\epsilon}}
\newcommand{\buen}{\bu_{\epsilon_n}}
\newcommand{\param}{\lambda}
\newtheorem{theorem}[equation]{Theorem}
\newtheorem{lemma}[equation]{Lemma}

\newtheorem{Proposition}[equation]{Proposition}
\newtheorem{definition}[equation]{Definition}

\newtheorem{remark}[equation]{Remark}

\numberwithin{equation}{section}

\newcommand{\otimess}{\overset{s}{\otimes}}

\newcommand{\trap}[1]{{#1}_{\tau}}
\newcommand{\difp}[1]{d^+{#1}}
\newcommand{\tran}[1]{{#1}_{-\tau}}
\newcommand{\difn}[1]{d^-{#1}}
\newcommand{\bou}{\partial\Omega}

\begin{document}

\title[Regularity for systems with symmetric gradients]{Global regularity for systems with $p$-structure depending on the
  symmetric gradient} \author{Luigi
  C. Berselli } \address{Luigi C. Berselli, Dipartimento di
  Matematica, Universit\`a di Pisa, Via F.~Buonarroti~1/c, I-56127
  Pisa, ITALY.}  \email{luigi.carlo.berselli@unipi.it} \author
{Michael R\r u\v zi\v cka{}} \address{Michael R\r u\v zi\v cka{},
  Institute of Applied Mathematics, Albert-Ludwigs-University
  Freiburg, Eckerstr.~1, D-79104 Freiburg, GERMANY.}
\email{rose@mathematik.uni-freiburg.de}

\begin{abstract}
  In this paper we study on smooth bounded domains the global
  regularity (up to the boundary) for weak solutions to systems having
  $p$-structure depending only on the symmetric part of the gradient.
  \\[3mm]
  \\[3mm]
%  \\[3mm]
%  \today,\ \thistime %,\quad\input{pwd.tex}
\end{abstract}
\maketitle
%
%%%%%%%%%%%%%%%%%%%%%%%%%%%%%%%%%%%%%%%%
%
\section{Introduction}
\label{sec:intro}
In this paper we study regularity of weak solutions to the boundary
value problem 
\begin{equation}
  \label{eq}
  \begin{aligned}
    -\divo \bfS(\bfD\bu)&=\bff\qquad&&\text{in }\Omega,
    \\
    \bfu &= \bfzero &&\text{on } \partial \Omega,
  \end{aligned}
\end{equation}
where $\bfD\bu:=\tfrac 12(\nabla\bu+ \nabla\bu^\top)$ denotes the
symmetric part of the gradient $\nabla \bu$ and where\footnote{We
  restrict ourselves to the problem in three space dimensions, even if
  results can be easily transferred to the problem in $\setR^d$ for
  all $d\geq 2$.  } $\Omega\subset\setR^3$ is a bounded domain with a
$C^{2,1}$ boundary $\partial\Omega$.  Our interest in this system
comes from the $p$-Stokes system
\begin{equation}
  \label{stokes}
  \begin{aligned}
    -\divo \bfS(\bfD\bu) +\nabla \pi&=\bff\qquad&&\text{in }\Omega,
    \\
    \divo \bfu &= 0 &&\text{in } \Omega,
    \\
    \bfu &= \bfzero &&\text{on } \partial \Omega.
  \end{aligned}
\end{equation}
In both problems the typical example for $\bS$ we have in
mind is
\begin{equation*}
%  \label{eq:stress}
  \bfS(\bD \bu) = %\mu_0 \bD\bu +
  \mu (\delta+\abs{\bD \bu})^{p-2}\bD\bu\,, 
\end{equation*}
where $p \in\, (1,2]$, $\delta\geq 0$, and $\mu>0$.  In previous
investigations of~\eqref{stokes} only suboptimal results for the
regularity up to the boundary have been proved. Here we mean
suboptimal in the sense that the results are weaker than the results
known for $p$-Laplacian systems,
cf.~\cite{acerbi-fusco,gia-mod-86,giu1}.  Clearly, the
system~\eqref{eq} is obtained from~\eqref{stokes} by dropping the
divergence constraint and the resulting pressure gradient. Thus the
system~\eqref{eq} lies in between system~\eqref{stokes} and
$p$-Laplacian systems, which depend on the full gradient $\nabla \bu$.

We would like to stress that the system~\eqref{eq} is of own
independent interest, since it is studied within plasticity theory, when 
formulated in the framework of deformation theory
(cf.~\cite{fuse,SS00}).  In this context the unknown is the
displacement vector field $\bu=(u^1,u^2,u^3)^\top$, while the external
body force $\bff=(f^1,f^2,f^3)^\top$ is given. The stress tensor
$\bfS$, which is the tensor of small elasto-plastic deformations,
depends only on $\bfD\bu$. Physical interpretation and discussion of
both systems~\eqref{eq} and~\eqref{stokes} and the underlying models
can be found, e.g., in~\cite{bird,fuse,kach,ma-ra-model,NH80}.

We study global regularity properties of weak solutions to~\eqref{eq}
in sufficiently smooth and bounded domains $\Omega$; we obtain for all
$p \in (1,2]$ the optimal result, namely that $\bF(\bD \bu)$ belongs
to $ W^{1,2}(\Omega)$, where the nonlinear tensor-valued function
$\bF$ is defined in~\eqref{eq:def_F}. This result has been proved near
a flat boundary in~\cite{SS00} and is the same result as for
$p$-Laplacian systems (cf.~\cite{acerbi-fusco,gia-mod-86,giu1}).  The
situation is quite different for~\eqref{stokes}. There the optimal
result, i.e.~$\bF(\bD \bu) \in W^{1,2}(\Omega)$, is only known for
{\it(i)} two-dimensional bounded domains (cf.~\cite{KMS2} where even
the $p$-Navier-Stokes system is treated), {\it(ii)} the space-periodic
problem in $\setR^d$, $d\ge2$, which follows immediately from interior
estimates, i.e.~$\bF(\bD \bu) \in W^{1,2}_\loc (\Omega)$, which are
known in all dimensions and the periodicity of the solution, 
{\it(iii)} if the no-slip boundary condition is replaced by perfect
slip boundary conditions (cf.~\cite{KT13}), and {\it(iv)} in the case of small $\bff$
(cf.~\cite{CG2008}). We also observe that the above results for the
$p$-Stokes system (apart
those in the space periodic setting) require the
stress tensor to be non-degenerate, that is $\delta>0$.
In the case of homogeneous Dirichlet boundary conditions and three- and higher-dimensional
bounded, sufficiently smooth domains only suboptimal results are
known. To our knowledge the state of the art for general data is that
$\bF(\bD \bu) \in W^{1,2}_\loc (\Omega)$, tangential derivatives of
$\bF(\bD\bu)$ near the boundary belong to $L^2$, while the normal
derivative of $\bF(\bD\bu)$ near the boundary belongs to some $L^{q}$,
where $q=q(p)<2$ (cf.~\cite{hugo-petr-rose,br-reg-shearthin} and the
discussion therein). We would also like to mention a result for
another system between~\eqref{stokes} and $p$-Laplacian system, namely
if \eqref{stokes} is considered with $\bS$ depending on the full velocity
gradient $\nabla \bu$. In this case it is proved in \cite{CM15} that
$\bu \in W^{2,r}(\setR^3)\cap W^{1,p}_0(\setR^3)$ for some $r>3$,
provided $p<2$ is very close to $2$.

In the present paper we extend to the general case of bounded sufficiently smooth
domains and to possibly degenerate stress tensors, that is the case $\delta=0$, the optimal regularity result
for~\eqref{eq} of Seregin and Shilkin~\cite{SS00} in the case of a
flat boundary.  
% perfect plasticity, where optimal results are proved, but which
% requires nevertheless a flat portion of boundary. In this paper we
% extend the results to cover a general smooth domain, by combining
% several tools present in literature and by also modifying the proof in
% order to keep it elementary as well as flexible to further
% applications to related problems. 
The precise result we prove is the following:
% Now we can formulate our main results concerning the regularity
% properties of weak solutions to problems~\eqref{eq} in general smooth
% domains.
%
\begin{theorem}
  \label{thm:MT}
  Let the tensor field $\T$ in~\eqref{eq} have
  $(p,\delta)$-structure for some $p \in (1,2]$, and
  $\delta \in [0,\infty)$, and let $\bF$ be the associated tensor
  field to $\bS$. Let $\O\subset\eR^3$ be a bounded domain with
  $C^{2,1}$ boundary and let $\ff\in L^{p'}(\O)$. Then, the unique
  weak solution $\bu\in W^{1,p}_{0}(\Omega)$ of the
  problem~\eqref{eq} satisfies
  \begin{equation*}
    \begin{aligned}
      \intO   \abs{\nabla \F(\bD\bu)}^2\,d\bx
      &\le c \,, %(\norm{\ff}_{p'},\delta, \partial \Omega) \,,
    \end{aligned}
  \end{equation*}  
  where $c$ denotes a positive function which is non-decreasing in
  $\|\ff\|_{p'}$ and $\delta$, and which depends on %the boundary of
  the domain through its measure $\abs{\Omega}$ and the $C^{2,1}$-norms of the local description of
  $\partial \Omega$.  In particular, the above estimate implies that
  $\bu\in W^{2,\frac{3p}{p+1}}(\Omega).  $
\end{theorem}
%\comment{this theorem also holds for $\delta=0$ using the
%  approximation procedure from \cite{bdr-7-5}. we have to check }%
%
%%%%%%%%%%%%%%%%%%%%%%%%%%%%%%%%%%%%%%%%
%
\section{Preliminaries and main results}
\label{sec:prep} 
In this section we introduce the notation we will use, state the precise assumptions on
the extra stress tensor $\bS$, and formulate the main results of the paper.
\subsection{Notation}
We use $c, C$ to denote generic constants, which may change from line
to line, but are independent of the crucial quantities. Moreover, we
write $f\sim g$ if and only if there exists constants $c,C>0$ such
that $c\, f \le g\le C\, f$. In some cases we need to specify the
dependence on certain parameters, and consequently we denote by
$c(\,.\,)$ a positive function which is non-decreasing with respect to
all its arguments.

We use standard Lebesgue spaces $(L^p(\Omega),\,\|\,.\,\|_{p})$ and
Sobolev spaces $(W^{k,p}(\Omega),$ $ \,\|\,.\,\|_{k,p})$, where
$\Omega \subset \setR^3$, is a sufficiently smooth bounded domain. The
space $W^{1,p}_0(\Omega)$ is the closure of the compactly supported,
smooth functions $C^\infty_0 (\Omega)$ in $W^{1,p}(\Omega)$. Thanks to
the \Poincare{} inequality we equip $W^{1,p}_0(\Omega)$ with the
gradient norm $\norm{\nabla\, \cdot\,}_p$.  When dealing with
functions defined only on some open subset $G\subset \Omega$, we
denote the norm in $L^p(G)$ by $\|\,.\,\|_{p,G}$. As usual
we use the symbol $\rightharpoonup$ to denote weak convergence, and
$\rightarrow$ to denote strong convergence. The symbol $\spt f$
denotes the support of the function $f$. We do not distinguish between
scalar, vector-valued or tensor-valued function spaces. However, we
denote vectors by boldface lower-case letter as e.g.~$\bu$ and tensors
by boldface upper case letters as e.g.~$\bS$. For vectors
$\bu, \bv \in \eR^3$ we denote
$\bu\otimess \bv:=\frac 12 (\bu\otimes \bv + (\bu\otimes \bv)^\top)$,
where the standard tensor product $\bu\otimes \bv\in \eR^{3 \times 3}$
is defined as $(\bu\otimes \bv)_{ij}:=u_iv_j$.  The scalar product of
vectors is denoted by $\bu\cdot \bv= \sum_{i=1}^3u_i v_i$ and the
scalar product of tensors is denoted by
$\bA\cdot\bB:=\sum_{i,j=1}^3A_{i j}B_{i j}$.

Greek lower-case letters take only values $1, 2$, while Latin
lower-case ones take the values $1, 2, 3$. We use the summation
convention over repeated indices only for Greek lower-case letters,
but not for Latin lower-case ones.

\subsection{$(p,\delta)$-structure}% Basic properties of the extra stress tensor} 
\label{sec:stress_tensor}
We now define what it means that a tensor field $\bS$ has
$(p,\delta)$-structure, see~\cite{die-ett, dr-nafsa}. For a tensor
$\bfP \in \setR^{3 \times 3} $ we denote its symmetric part by
$\bP^\sym:= \frac 12 (\bP +\bP^\top) \in \setR_\sym ^{3 \times 3}:=
\set {\bfP \in \setR^{3 \times 3} \,|\, \bP =\bP^\top}$. % The scalar
% product between two tensors $\bP, \bQ$ is denoted by $\bP\cdot \bQ$,
% and
We use the notation $\abs{\bP}^2=\bP \cdot \bP $.

It is convenient to define for $t\geq0 $ a special
N-function\footnote{For the general theory of N-functions and Orlicz
  spaces we refer to \cite{ren-rao}. }
$\phi(\cdot)=\phi_{p,\delta}(\cdot)$, for $p \in (1,\infty)$, $\delta\ge
0$, by
\begin{equation} 
  \label{eq:5}
  \varphi(t):= \int _0^t (\delta +s)^{p-2} s  \, ds\,.
\end{equation}
The function $\phi$ satisfies, uniformly in $t$ and independently of
$\delta$, the important equivalence
\begin{align}
  \label{eq:equi1}
  \phi''(t)\, t &\sim \phi'(t)\,,
  \\
  \label{eq:equi2}
  \phi'(t)\, t &\sim \phi(t)\,,
  \\
  t^p+\delta^p &\sim \phi(t) +\delta^p\,. \label {eq:equi3}
\end{align}
We use the convention that if $\phi''(0)$ does not exist, the
left-hand side in~\eqref{eq:equi1} is continuously extended by zero
for $t= 0$.  We define the {\rm shifted
  N-functions} $\set{\phi_a}_{a \ge 0}$,
cf. ~\cite{die-ett,die-kreu,dr-nafsa}, for $t\geq0$ by
\begin{equation*}
%  \label{eq:phi_shifted}
  \phi_a(t):= \int _0^t \frac { \phi'(a+s)\,s}{a+s}\, ds  
\end{equation*}
Note that the family $\set{\phi_a}_{a \ge 0}$ satisfies the $\Delta_2$-condition
uniformly with respect to ${a \ge 0}$, i.e.~$\phi_a(2t) \leq c(p)
\phi_a(t)$ holds for all $t\ge 0$.
\begin{definition}[$(p,\delta)$-structure]
\label{ass:1}
  We say that a tensor field $\bS\colon \setR^{3 \times 3} \to
  \setR^{3 \times 3}_\sym $ belonging to $C^0(\setR^{3 \times
    3},\setR^{3 \times 3}_\sym )\cap C^1(\setR^{3 \times 3}\setminus
  \{\bfzero\}, \setR^{3 \times 3}_\sym ) $, satisfying $\bS(\bP) =
  \bS\big (\bP^\sym \big )$, and $\bS(\mathbf 0)=\mathbf 0$ possesses
  {\rm $(p,\para)$-structure}, if for some $p \in (1, \infty)$,
  $\para\in [0,\infty)$, and the N-function
  $\varphi=\varphi_{p,\delta}$ (cf.~\eqref{eq:5}) there exist
  constants $\kappa_0, \kappa_1 >0$ such that
   \begin{equation}
     \label{eq:ass_S}
     \begin{aligned}
       \sum\limits_{i,j,k,l=1}^3 \partial_{kl} S_{ij} (\bP)
       Q_{ij}Q_{kl} &\ge \kappa_0
       \,\phi''(|\bP^\sym|)|\bQ^\sym|^2\,,%\label{1.4b}
       \\[-2mm]
       \big |\partial_{kl} S_{ij}({\bP})\big | &\le \kappa_1 \,\phi''
       (|\bP^\sym|)%\label{1.5b}
     \end{aligned}
   \end{equation} 
   are satisfied for all $\bP,\bQ \in \setR^{3\times 3} $ with
   $\bP^\sym \neq \bfzero$ and all $i,j,k,l=1,2,3$.  The constants
   $\kappa_0$, $\kappa_1$, and $p$ are called the {\em
     characteristics} of $\bfS$.
\end{definition}
\begin{remark}
  \label{rem:delta0}
  {\rm (i) Assume that $\bS$ has $(p,\delta)$-structure for some
    $\delta \in [0,\delta_0]$. Then, if not otherwise stated, the
    constants in the estimates depend only on the characteristics of
    $\bfS$ and on $\delta_0$, but are independent of $\delta$. 
 
    (ii) An important example of a tensor field $\bfS$ having
    $(p,\delta)$-structure is given by $ \bfS(\bfP) =
    \phi'(\abs{\bfP^\sym})\abs{\bfP^\sym}^{-1} \bfP^\sym$.  In this
    case the characteristics of $\bfS$, namely $\kappa_0$ and
    $\kappa_1$, depend only on $p$ and are independent of $\delta \geq
    0$.  

    (iii) For a tensor field $\bS$ with $(p,\delta)$-structure we have
    $\partial_{kl} S_{ij} (\bP)= \partial_{kl} S_{ji} (\bP)$, for all
    $i,j,k,l =1,2,3$ and all $\bP \in \setR^{3\times 3}$, due to its
    symmetry. Moreover, from \linebreak $\bS(\bP) = \bS\big (\bP^\sym \big )$
    follows
    $\partial_{kl} S_{ij} (\bP)= \frac 12 \partial_{kl} S_{ij}
    (\bP^\sym) + \frac 12 \partial_{lk} S_{ij} (\bP^\sym)$, for all
    $i,j,k,l =1,2,3$ and all $\bP \in \setR^{3\times 3}$, and
    consequently $\partial_{kl} S_{ij} (\bP)= \partial_{lk} S_{ij}(\bP)$
    for all $i,j,k,l =1,2,3$ and all $\bP \in \setR_\sym^{3\times
      3}$.}
\end{remark}
To a tensor field $\bS$ with $(p,\delta)$-structure we associate
the tensor field $\bF\colon\setR^{3 \times 3} \to \setR^{3 \times
  3}_\sym$ defined through
\begin{align}
  \label{eq:def_F}
  \bF(\bP):= \big (\para+\abs{\bP^\sym} \big )^{\frac
    {p-2}{2}}{\bP^\sym } \,.
\end{align}
The connection between $\bfS$, $\bfF$, and $\set{\phi_a}_{a \geq 0}$
is best explained by the following proposition
(cf.~\cite{die-ett}, \cite{dr-nafsa}). 
\begin{Proposition}%[See Lemma~2.3 in~\cite{DieE08}]
  \label{lem:hammer}
  Let $\bfS$ have $(p,\delta)$-structure, and let $\bfF$ be defined
  in~\eqref{eq:def_F}.  Then
  \begin{subequations}
    \label{eq:hammer}
    \begin{align}
      \label{eq:hammera}
      \big({\bfS}(\bfP) - {\bfS}(\bfQ)\big) \cdot \big(\bfP-\bfQ \big)
      &\sim \bigabs{ \bfF(\bfP) - \bfF(\bfQ)}^2\,,
      \\
      &\sim \phi_{\abs{\bfP^\sym}}(\abs{\bfP^\sym - \bfQ^\sym})\,,
      \label{eq:hammerb}
      \\
      \label{eq:hammerc}
      &\sim \phi''\big( \abs{\bfP^\sym} + \abs{\bfP^\sym - \bfQ^\sym}
      \big)\abs{\bfP^\sym - \bfQ^\sym}^2\,,
      \\
      \abs{\bfS(\bfP) - \bfS(\bfQ)} &\sim
      \phi'_{\abs{\bfP^\sym}}\big(\abs{\bfP^\sym -
        \bfQ^\sym}\big)\,,  \label{eq:hammere}  
      \intertext{uniformly in $\bfP, \bfQ \in \setR^{3 \times 3}$.
        Moreover,  uniformly in $\bfQ \in \setR^{3 \times 3}$,} 
      \label{eq:hammerd}
      \bfS(\bfQ) \cdot \bfQ &\sim \abs{\bfF(\bfQ)}^2 \sim
      \phi(\abs{\bfQ^\sym}).
    \end{align}
  \end{subequations}
  The constants depend only on the characteristics of $\bfS$.
\end{Proposition} 
For a detailed discussion of the properties of $\T$ and $\F$ and their
relation to Orlicz spaces and N-functions we refer the reader
to~\cite{dr-nafsa,bdr-7-5}. Since in the following we shall insert into $\T $ and $\F$ only
symmetric tensors, we can drop in the above formulas the superscript
``$^\sym $'' and restrict the admitted tensors to symmetric ones.
%
%\subsection{On some equivalent quantities}
%

We recall that the following equivalence, which is proved
in~\cite[Lemma~3.8]{bdr-7-5}, 
\begin{equation}
  \label{eq:grad}
  |\partial_i\bF(\bQ)|^2\sim\phi''(|\bQ|)|\partial_i\bQ|^2, 
\end{equation}
valid for all smooth enough symmetric tensor fields
$\bQ\in \setR_\sym ^{3 \times 3}$. The proof of this equivalence is
based on Proposition~\ref{lem:hammer}. This Proposition and the theory
of divided differences also imply (cf.~\cite[(2.26)]{br-reg-shearthin})
that 
\begin{equation}
  \label{eq:tang}
  |\partial_\tau \bF(\bQ)|^2\sim\phi''(|\bQ|)|\partial_\tau\bQ|^2
\end{equation}
for all smooth enough symmetric tensor fields
$\bQ\in \setR_\sym ^{3 \times 3}$.

A crucial observation in~\cite{SS00} is that the quantities
in~\eqref{eq:grad} are also equivalent to several further
quantities. To formulate this precisely we introduce for $i=1,2,3$ and
for sufficiently smooth symmetric tensor fields $\bQ$ the quantity
\begin{equation}
  \label{eq:P}
  \mathcal{P}_i(\bQ):=\partial_i\bS(\bQ)\cdot\partial_i\bQ =
  \sum_{k,l,m,n=1}^3\partial_{kl}S_{mn}(\bQ) \,\partial_i Q_{kl}\,\partial_i Q_{mn}\,.
\end{equation}
Recall, that in the definition of $\mathcal{P}_i(\bQ)$ there is no
summation convention over the repeated Latin lower-case index $i$ in
$\partial_i\bS(\bQ)\cdot\partial_i\bQ$. Note that if $\bS$ has
$(p,\delta)$-structure, then $ \mathcal{P}_i(\bv)\geq0$, for
$i=1,2,3$.
%
%  and in particular the relevant
% quantity will be $ \mathcal{P}_3(\bu)$, since it is connected with the
% information in the normal direction. This quantity is estimated here
% directly in terms of previously known tangential derivatives, and not
% extracting information from the system.
%
There hold the following important equivalences, first proved in~\cite{SS00}:
\begin{Proposition}
\label{prop:equivalence}
Assume that $\bS $ has $(p,\delta)$-structure. Then the following equivalences are valid,
for all smooth enough symmetric tensor fields $\bQ$ and all $i=1,2,3$
  \begin{align}
    &\mathcal{P}_i(\bQ)\sim\phi''(|\bQ|)|\partial_i
    \bQ|^2\sim|\partial_i \bF(\bQ)|^2\,,\label{eq:p-F}
    % \\
    % &|\partial_i \bS(\bD\bv)|^2\sim \phi''(|\bD\bv|)  \mathcal{P}_i(\bv),
    \\
    &\mathcal{P}_i(\bQ)\sim\frac{|\partial_i\bS(\bQ)|^2}{\phi''(|\bQ|)}\,,\label{eq:p-S}
%      \sim |\partial_i \bF(\bv)|^2.
\end{align}
with constants only depending on the characteristics of $\bS$.
\end{Proposition}
\begin{proof}
  The assertions are proved in~\cite{SS00} using a different
  notation. For the convenience of the reader we sketch the proof
  here.  The equivalences in~\eqref{eq:p-F} follow
  from~\eqref{eq:grad}, \eqref{eq:P} and the fact that $\bS$ has
  $(p,\delta)$-structure. Furthermore, we have, using~\eqref{eq:p-F},
  \begin{align*}
    \abs{\mathcal P_i(\bQ)}^2 
    &\le \abs{\partial _i\bS(\bQ)}^2 \abs{\partial _i \bQ}^2 
     \le c\,\abs{\partial _i\bS(\bQ)}^2 \frac
      {\mathcal{P}_i(\bQ)}{\phi''(|\bQ|)}\,,
  \end{align*}
  which proves one inequality of~\eqref{eq:p-S}. The other follows
  from
  \begin{align*}
    \abs{\partial _i\bS(\bQ)}^2 
    &\le c\sum_{k,l=1}^3 \abs{\partial_{kl}\bS(\bQ)\, \partial _i
      Q_{kl}}^2  \le c \,\big(\phi''(\abs{\bQ})\big )^2 \abs{\partial
      _i \bQ}^2 \le c \,\phi''(\abs{\bQ}) \mathcal P _i (\bQ)\,,
  \end{align*}
where we used~\eqref{eq:ass_S} and~\eqref{eq:p-F}.
\end{proof}
\subsection{Existence of weak solutions} 
In this section we define weak solutions of~\eqref{eq}, recall the main
results of existence and uniqueness and discuss a perturbed problem,
which is used to justify the computations that follow. From now on we
restrict ourselves to the case $p\le 2$.
\begin{definition}
  We say that $\bu \in W^{1,p}_0(\Omega)$ is a weak solution
  to~\eqref{eq} if for all $\bv\in W^{1,p}_0(\Omega)$ 
  \begin{equation*}
    \int_\Omega \bS(\bD\bu)\cdot\bD\bv\,d\bx=\int_\Omega \bff\cdot
    \bv\,d\bx\,. 
  \end{equation*}
\end{definition}
We have the following very standard result:
\begin{Proposition}
\label{thm:existence}
  Let the tensor field $\T$ in~\eqref{eq} have
  $(p,\delta)$-structure for some $p \in (1,2]$, and $\delta \in
  [0,\infty)$. Let $\O\subset\eR^3$ be a bounded domain with $C^{2,1}$
  boundary and let $\ff\in L^{p'}(\O)$. Then, there exists a unique
  weak solution $\bu$ to~\eqref{eq} such that
  \begin{equation*}
%    \label{eq:main-apriori-estimate}
    \int_\Omega\phi(|\bD\bu|)\,d\bx\leq  c(\|\bff\|_{p'},\delta)\,.
  \end{equation*}
\end{Proposition}
\begin{proof}
  The assertions follow directly from the assumptions, by using the
  theory of monotone operators.
\end{proof}
In order to justify some of the following computations we find it convenient to consider a
perturbed problem, where we add to the tensor field $\bS$ with $(p,\delta)$-structure a
linear perturbation. Using again the theory of monotone operators one can easily prove:
% This result seems rather standard but since we need to
% use it in a very precise version, in the sequel we sketch the proof of
% the relevant details.
%
\begin{Proposition}
\label{thm:existence_perturbation}
Let the tensor field $\T$ in~\eqref{eq} have $(p,\delta)$-structure
for some $p \in (1,2]$, and $\delta \in [0,\infty)$ and let
$\bff\in L^{p'}(\Omega)$ be given. Then, there exists a unique weak
solution $\bue\in W^{1,2}_0(\Omega)$ of the problem 
\begin{equation}
  \label{eq-e}
  \begin{aligned}
    -\divo \bfS^\vep (\bfD\bue)&=\bff\qquad&&\text{in
    }\Omega\,,
    \\
    \bue &= \bfzero &&\text{on } \partial \Omega\,,
  \end{aligned}
\end{equation}
where 
  \begin{equation*}
%    \label{eq:perturbed_S}
    \bS^\epsilon(\bQ):=\epsilon \,\bQ +
    \bS(\bQ),\qquad\text{with }\epsilon>0\,,
  \end{equation*}
  i.e.~ $\bue$ satisfies for all $\bv \in W^{1,2}_0(\Omega)$
  \begin{equation*}
    \int_\Omega \bS^\vep(\bD\bue)\cdot\bD\bv\,d\bx=\int_\Omega \bff\cdot
    \bv\,d\bx\,. 
  \end{equation*}
  The solution $\bue$ satisfies the estimate
  \begin{equation}
    \label{eq:main-apriori-estimate2}
    \epsilon\int_\Omega
    |\nabla\bue|^2+\phi(|\bD\bue|)\,d\bx\leq
    c(\|\bff\|_{p'},\delta).
  \end{equation}
\end{Proposition}
\begin{remark}
  In fact, one could already prove more at this point. Namely, that for
  $\epsilon\to0$, the unique solution $\bue$ converges to the unique
  weak solution $\bu $ of the unperturbed problem~\eqref{eq}. Let us
  sketch the argument only, since later we get the same result with
  different easier arguments. From \eqref{eq:main-apriori-estimate2}
  and the properties of $\bS$ follows that
  \begin{align*}
    \bue &\rightharpoonup \bu \qquad\text{in }W^{1,p}_0(\Omega)\,,
    \\
    \bS(\bD\bue)&\rightharpoonup \boldsymbol\chi \qquad \text{in }L^{p'}(\Omega)\,.
  \end{align*}
  Passing to the limit in the weak formulation of the perturbed problem, we get
  \begin{equation*}
    \int_\Omega\boldsymbol {\chi}\cdot\bD\bv\,d\bx
    =\int_\Omega\bff\cdot\bv\,d\bx\qquad\forall\,\bv \in
    W^{1,p}_0(\Omega)\,. 
  \end{equation*}
  % The identification $\boldsymbol\chi=\bS(\bD\bu)$ follows e.g.~from the
  % point-wise convergence of $\bD\bue$ to $\bD\bu$. 
  One can not show directly that
  $\lim _{\vep \to 0} \int_\Omega \vep \bD\bue \cdot (\bD\bue
  -\bD\bu)\, d\bx =0$, since $\bD\bu $ belongs to $L^p(\O)$ only.
  Instead one uses the Lipschitz truncation method
  (cf.~\cite{dms,r-cetraro}). Denoting by $\bv^{\vep,j}$ the Lipschitz
  truncation of $\xi(\bue-\bu)$, where $\xi\in C^\infty_0(\Omega)$ is
  a localization, one can show, using the ideas from
  \cite{dms,r-cetraro}, that
\begin{equation}
\label{eq:convergence}
  \limsup_{\vep\to 0}\Big|\int_\Omega \big (\bS(\bD\bue)
    -\bS(\bD\bu)\big )\cdot \bD\bv^{\vep,j}\,d\bx \Big|=0,
\end{equation}
which implies $\bD\bue\to\bD\bu$ almost everywhere in $\Omega$.
Consequently, we have $\boldsymbol\chi=\bS(\bD\bu)$, since weak and
a.e.~limits coincide.
\end{remark}
% \begin{proof}

% In particular, to show~\eqref{eq:convergence} we write 
% %
% \begin{equation*}
%   \begin{aligned}
%    & \langle\bS(\bD\buen) -\bS(\bD\bu),\bD\bv^{n,j}\rangle
%     \\
%     &=\langle\bS(\bD\buen)+\epsilon_n\bD\buen
%     -\bS(\bD\bu)-\epsilon_n\bD\bue,\bD\bv^{n,j}\rangle-\epsilon_n\langle\bD\buen
%     -\bD\bu,\bD\bv^{n,j}\rangle, 
%   \end{aligned}
% \end{equation*}
% %
% and we observe that, by the a priori
% estimate~\eqref{eq:main-apriori-estimate2}, one has
% \begin{equation*}
%   \lim_{n\to+\infty}\epsilon_n\langle(\bD\buen-\bD\bu)\cdot\bD\bv^{n,j}\rangle=0,
% \end{equation*}
% since
% $\epsilon_n\big|\langle\bD\buen\cdot\bD\bv^{n,j}\rangle\big|\leq\sqrt{\epsilon_n}
% \|\sqrt{\epsilon_n}\bD\buen\|_2\|\bD\bv^{n,j}\|_2\leq\sqrt{\epsilon_n}\,c(\|\bff\|_{p'})$. The 
% other terms go easily to zero by the definition of weak convergence
% and by the properties of the truncation operator.
% \end{proof}
%
\subsection{Description and properties of the boundary}
\label{sec:bdr} 
$\hphantom{}$
% \comment{We did not stress in the previous paper, that this local
%   coordinates require orthogonal transformations and if the equations
%   are not invariant, this produces additional terms}
% \comment{why they should not be invariant?}
We assume that the boundary $\partial\O$ is of class $C^{2,1}$, that
is for each point $P\in\partial\O$ there are local coordinates such
that in these coordinates we have $P=0$ and $\partial\O$ is locally
described by a $C^{2,1}$-function, i.e.,~there exist
$R_P,\,R'_P \in (0,\infty),\,r_P\in (0,1)$ and a $C^{2,1}$-function
$a_P:B_{R_P}^{2}(0)\to B_{R'_P}^1(0)$ such that
\begin{itemize}
\item   [\rm (b1) ] $\bx\in \partial\O\cap (B_{R_P}^{2}(0)\times
  B_{R'_P}^1(0))\ \Longleftrightarrow \ x_3=a_P(x_1,x_2)$,
\item   [\rm (b2) ] $\Omega_{P}:=\{(x,x_{3})\fdg x=(x_1,x_2)^\top
 \in  B_{R_P}^{2}(0),\ a_P(x)<x_3<a_P(x)+R'_P\}\subset \Omega$, 
%\item   [\rm (b2) ] $\bx\in\O_P:=\O\cap (B_{R_P}^{2}(0)\times
%B_{R_P}^1(0))\ \Longleftrightarrow \   x_3>a_P(x_1,x_2)$,
\item [\rm (b3) ] $\nabla a_P(0)=\bfzero,\text{ and }\forall\,x=(x_1,x_2)^\top
  \in B_{R_P}^{2}(0)\quad |\nabla a_P(x)|<r_P$,
\end{itemize}
where $B_{r}^k(0)$ denotes the $k$-dimensional open ball with center
$0$ and radius ${r>0}$.  Note that $r_P $ can be made arbitrarily
small if we make $R_P$ small enough.  In the sequel we will also use,
for $0<\lambda<1$, the following scaled open sets, $\lambda\,
\Omega_P\subset \Omega_P$ defined as follows
\begin{equation*}
%  \label{eq:scaled_omega_P}
  \lambda\, \Omega_P:=\{(x,x_{3})\fdg x=(x_1,x_2)^\top
 \in
  B_{\lambda R_P}^{2}(0),\ a_P(x)<x_3<a_P(x)+\lambda R_P'\}.
\end{equation*}
To localize near to $\partial\Omega\cap \partial\Omega_P$, for $P\in\bou$, we fix smooth
functions $\xi_{P}:\setR^{3}\to\setR$ such that %$0\leq\xi_P\leq1$ and
\begin{itemize}
\item [$\rm (\ell 1)$] $\chi_{\frac{1}{2}\Omega_P}(\bx)\leq\xi_P(\bx)\leq
  \chi_{\frac{3}{4}\Omega_P}(\bx)$,
\end{itemize}
where $\chi_{A}(\bx)$ is the indicator function of the measurable set
$A$. 
%\begin{equation*}
% \hspace*{-30mm}{\rm (\ell 1)}\quad  \xi_P=\left\{
%    \begin{aligned}
%      &1\quad\text{ for }x\in\Omega_P\cap \big(B^2_{{R_P}/2}(0)\times
%      B^1_{{R_P}/2}(0)\big), 
%    \\
%    & 0\quad\text{ for }x\in\Omega_P\backslash
%    \big(B^2_{3{R_P}/4}(0)\times B^1_{3{R_P}/4}(0)\big).
%  \end{aligned}
%\right. 
%\end{equation*}
For the remaining interior estimate we  localize by a smooth function
${0\leq\xi_{00}\leq 1}$ with $\spt \xi_{00}\subset\Omega_{00}$,
where $\Omega_{00}\subset \Omega$ is an open set such that
$\dist(\partial\Omega_{00},\,\partial\Omega)>0$.  
%We first perform calculations in $\Omega_{P}$ for all
%$P\in \partial\Omega$, especially in the study of regularity of
%tangential derivatives.
%The local estimates near the boundary are obtained in two steps. In
%the first one (see Sections~\ref{sec:tan} and~\ref{sec:tanp}) we
%estimate in $\Omega_P$ only tangential derivatives as defined
%below. In the second one we use the new obtained information and
%compute the normal derivatives from the system. 
Since the boundary $\bou $ is compact, we can use an appropriate
finite sub-covering which, together with the interior estimate, yields
the global estimate.
%  and we use the covering of
% $\partial\Omega$ made with the open sets
% $\{\O_P\}_{P\in\partial\Omega}$. 

Let us introduce the tangential derivatives near the boundary. To
simplify the notation we fix $P\in \bou$, $h\in (0,\frac{R_P}{16})$,
and simply write $\xi:=\xi_P$, $a:=a_P$. We use the standard notation
$\bx =(x',x_3)^\top$ and denote by $\be^i,i=1,2,3$ the canonical
orthonormal basis in $\setR^3$. In the following lower-case Greek
letters take values $1,\, 2$. For a function $g$ with $\spt
g\subset\spt\xi$ we define for $\alpha=1,2$
%positive and negative tangential translations:
%
\begin{equation*}
\begin{aligned}
  \trap g(x',x_3) = g_{\tau _\alpha}(x',x_3)&:=g\big (x' +
  h\,\be^\alpha,x_3+a(x'+h\,\be^\alpha)-a(x')\big )\,,
%  \\
%  \tran g(x',x_3)&:=g\big (x' - h\,\be^\alpha,x_3+a(x' -
%  h\,\be^\alpha)-a(x')\big )\,;
\end{aligned}
\end{equation*}
%%
%tangential differences
%%
%\begin{equation*}
%  \Delta^+ g:=\trap g-g,\qquad\Delta^- g:=\tran g-g\,;
%\end{equation*}
%%
and if $\Delta^+ g:=\trap g-g$, we define tangential divided
differences by
%
%\begin{equation*}
  $\difp g:= h^{-1}\Delta^+ g$.  
%\end{equation*}
%
It holds that, if $ g \in W^{1,1}(\O)$, then we have for $\alpha=1,2$
\begin{align} 
  \label{eq:1}
  \difp g \to \td g=\partial _{\tau_\alpha}g :=\partial_\alpha g +\partial_\alpha
  a\, \partial_3 g  \qquad \text{ as } h\to 0,
\end{align} 
almost everywhere in $\spt\xi$, (cf.~\cite[Sec.~3]{mnr2}). Conversely uniform $L^q$-bounds
for $\difp g$ imply that $\partial _\tau g $ belongs to $L^q(\spt\xi)$.
%Moreover, we have for all %
%%\marginpar{L:Here is $W^{1,q}(\Omega)$??, since we apply it also to
%%functions not zero at the boundary\\ M: I think that is ok}
% $1<q<\infty$, $ g \in W^{1,q}(\O)$ and all sufficiently small
%$h>0$, that 
%\begin{align}
%  \label{eq:2}
%  \exists\,c(a)>0:\quad \norm{\difp g }_{q,\spt\xi} \le
%  c(a)\norm{\nabla g }_q .
%\end{align}
%Conversely, if $\norm{\difp g  }_{q,\spt\xi} \le C$ for all
%sufficiently small $h>0$, then 
%\begin{align}
%  \label{eq:2a}
%  \norm {\partial _\tau  g }_{q,\spt\xi} \le C.
%\end{align}
%

%Now we formulate some auxiliary lemmas related to these objects. The
%first lemma clarifies the fact that tangential translations and
%tangential differences do not commute with partial
%derivatives. Also the explicit expressions can be used to
%quantitatively estimate the so called commutation terms, as called in
%turbulence theory~\cite{Ber-Gri-John2007}.
For simplicity we denote $\nabla a:=(\partial_1a,\partial_{2}a,
0)^\top$.
%
%and use the operations $\trap {(\cdot)}$, $\tran {(\cdot)}$,
%$\Delta^+(\cdot) $, $\Delta^+(\cdot) $, $\difp {(\cdot)}$ and $\difn
%{(\cdot)}$ also for vector-valued and tensor-valued functions,
%intended as acting component-wise.
%%
%\begin{lemma}
%  \label{lem:TD1} 
%  Let $\bv\in W^{1,1}(\Omega)$ such that $\spt \bv
%  \subset\spt\xi$. Then 
%%
%\begin{equation*}
%\begin{aligned}
%  \nabla\difpm \bv &=\difpm{\nabla \bv }+\trap{(\partial_3 \bv
%    )}\otimes\difpm{\nabla a},
%  \\
%  \bD\difpm \bv &=\difpm{\bD \bv }+\trap{(\partial_3 \bv
%    )}\otimess\difpm{\nabla a},
%  \\
%  \diver\difpm \bv &=\difpm\diver \bv +\trapm{(\partial_3 \bv
%    )}\difpm{\nabla a}
%  \\
%  \nabla \bv _{\pm\tau} &= (\nabla \bv )_{\pm\tau} + \trapm{(\partial_3 \bv
%    )}\difpm{\nabla a},
%\end{aligned}
%\end{equation*}
%where $\otimess$ is defined component-wise also for scalar and tensor-valued functions.
%\end{lemma}
%%
%The second lemma is devoted to the relation between tangential
%differences and tangential translations, provided that $h$ is small
%enough. 
%%
%\begin{lemma}
%  \label{lem:TD2}
%  Let $\spt g \subset\spt\xi$. Then
%\begin{equation*}
%  \trap{(\difn g )}=-\difp g ,\quad \tran{(\difp g )}=-\difn g , \quad 
%  \difn  g_\tau = - \difp g .
%\end{equation*}
%\end{lemma}
%%
The following variant of integration per parts will be often used.
%\marginpar{formulate differently with respect to the support!!}
%\marginpar{L: what to write?\\ M: I dont know what you mean.}
\begin{lemma}
  \label{lem:TD3}
  Let $\spt g\cup\spt f\subset\spt\xi$ and $h$ small enough. Then
  \begin{equation*}
    \intO f\tran g \, d\bx =\intO\trap f g\, d\bx.
  \end{equation*}
  Consequently, $\intO f\difp g \, d\bx= \intO(\difn f )g\, d\bx$.
  Moreover, if in addition $f$ and $g$ are smooth enough and at least
  one vanishes on $\partial\Omega$, then 
%\marginpar{added also    this} 
\begin{equation*} \intO f\td g \, d\bx= -\intO(\td f )g\,
    d\bx.
  \end{equation*}
\end{lemma}
\section{Proof of the main result}
In the proof of the main result we use finite differences to show
estimates in the interior and in tangential directions near the
boundary and calculations involving directly derivatives in
"{}normal"\ directions near the boundary. In order to justify that all
occurring quantities are well posed, we perform the estimate for the
approximate system~\eqref{eq-e}.

% we follow a strategy which combines
% finite differences, in order to locally estimate partial derivatives
% in the directions tangential to the boundary, together with
% calculations that handle directly quantities that are related to
% partial derivatives which are locally in the direction of $x_3$. These
% calculations are performed for the approximate system, i.e.~problem
% \eqref{eq} with $\bS$ replaced by $\bS^\vep$, to ensure that all
% occurring quantities are well defined. 

The first intermediate step is the following result for the
approximate problem. 
\begin{Proposition}  
  \label{prop:JMAA2017-1}
  Let the tensor field $\T$ in~\eqref{eq} have
  $(p,\delta)$-structure for some $p \in (1,2]$, and $\delta \in
  (0,\infty)$, and let $\bF$ be the associated tensor field to
  $\bS$. Let $\O\subset\eR^3$ be a bounded domain with $C^{2,1}$
  boundary and let $\ff\in L^{p'}(\O)$. Then, the unique weak solution
  $\bue\in W^{1,2}_{0}(\O)$ of the approximate problem~\eqref{eq-e}
  satisfies
%\comment{dependence on $a_P$?}%
  \begin{align}
    \begin{aligned}
      \epsilon \intO\xi_0^2 \abs{\nabla ^2\bue}^2+\xi_0^2
      \abs{\nabla \F(\bD\bue)}^2\,d\bx &\le c
      (\norm{\ff}_{p'},\norm{\xi_0}_{2,\infty},\delta) \,,
      \\
      \hspace*{-2mm} \epsilon \intO\xi^2_P \abs{\td \bD\bue}^2+\xi^2_P
      \abs{\td \F(\bD\bue)}^2\,d\bx &\le c
      (\norm{\ff}_{p'},\norm{\xi_P}_{2,\infty},\norm{a_P}_{C^{2,1}},\delta)
      \,. \hspace*{-3mm}
    \end{aligned} \label{est-eps}
  \end{align}
  Here $\xi_{0}$ is a cut-off function with support in the interior of
  $\Omega$, while for arbitrary $P\in \partial \Omega$ the function
  $\xi_{P}$ is a cut-off function with support near to the boundary
  $\partial \Omega$, as defined in Sec.~\ref{sec:bdr}.  The tangential
  derivative $\partial_\tau$ is defined locally in $\Omega_P$
  by~\eqref{eq:1}. Moreover, there exists a constant $C_1>0$ such that
  \begin{equation}
    \label{est-eps-1}
    \epsilon  \intO\xi^2 \abs{\partial _3 \bD\bue}^2+\xi^2
    \abs{\partial _3 \F(\bD\bue)}^2\,d\bx   
    \le c (\norm{\ff}_{p'},\norm{\xi}_{2,\infty},\norm{a}_{C^{2,1}},\delta^{-1},
    \vep^{-1},C_1)     
  \end{equation}
  provided that in the local description of the boundary there holds
  $r_P<C_1$ in $(b3)$.

In particular, these estimates imply that $\bue \in W^{2,2}(\Omega)$
and that~\eqref{eq-e} holds pointwise a.e.~in $\Omega$.
\end{Proposition}
The two estimates~\eqref{est-eps} are uniform with respect to $\vep$ and could be also
proved directly for the problem~\eqref{eq}. However, the third estimate~\eqref{est-eps-1}
depends on $\vep$ but is needed to justify all subsequent steps, which will give the
proof of an estimate uniformly in $\vep$, by using a different technique.

%  result of the above proposition will make the further
% manipulations of the solution $\bue$ totally justified. A limiting
% process will be finally used to show that uniform (with respect to
% $\epsilon$) estimates valid on $\bue$ are also valid for the solution
% of the original, unperturbed problem~\eqref{eq}.
%
\begin{proof}[Proof of Proposition~\ref{prop:JMAA2017-1}]
  The proof of estimate \eqref{est-eps} is very similar, being in fact
  a simplification (due to the fact that there is no pressure term
  involved), to the proof of the results in~\cite[Theorems 2.27,
  2.28]{br-reg-shearthin}. On the other hand the proof of
  \eqref{est-eps-1} is different from the one in
  \cite{br-reg-shearthin} due to the missing divergence constraint. In
  fact it adapts techniques known from nonlinear elliptic systems. For
  the convenience of the reader we recall the main steps
  here. % In particular, one can first deduce
  % estimates for the tangential derivatives and then obtain the missing
  % information on the derivative in the direction of $x_3$ by using the
  % equations point-wise.

  Fix $P\in \partial \Omega$ and use in $\O_P$ 
  \begin{equation*}
    \bv=\difn{(\xi^2\difp(\bue |_{\frac 12 {\Omega}_P}))}\\,
  \end{equation*}
  where $\xi:=\xi_P$, $a:=a_P$, and $h\in(0,\frac{R_{P}}{16})$, as a
  test function in the weak formulation of~\eqref{eq-e}. This yields 
\begin{equation*}
%  \label{eq:dtTx2}
  \begin{aligned}
    \intO& \xi^2\difp{\T^\vep(\bD\bue)}\cdot \difp \bD\bue \, d\bx
    \\
    =&-\intO \T^\vep(\bD\bue)\cdot\big(\xi ^2 \difp \partial _3 \bue
    -(\xi_{-\tau } \difn\xi +\xi \difn\xi) \partial_3\bue \big)
    \otimess\difn\nabla a\, d\bx
    \\
    &-\intO \T^\vep(\bD\bue)\cdot \xi^2 \trap{(\partial _3\bue)}\otimess
    \difn{\difp{\nabla a}} - \T^\vep(\bD\bue)\cdot\difn{\big(2\xi \nabla \xi
      \otimess \difp\bue\big)}\, d\bx
    \\
    &+\intO \T^\vep(\trap{(\bD\bue)})\cdot \big(2 \xi \partial_3\xi \difp\bue +
    \xi^2 \difp\partial_3\bue \big)\otimess\difp\nabla a \, d\bx
    \\
    &+\intO\ff\cdot\difn(\xi^2 \difp \bue)\, d\bx=:\sum_{j=1}^{8} I_j\,.
  \end{aligned}
\end{equation*}
From the assumption on $\bS$, Proposition~\ref{lem:hammer},
and~\cite[Lemma~3.11]{br-reg-shearthin} we have the following estimate
\begin{equation*}
  \begin{aligned}
    \label{eq:odhadT2}
    &\epsilon\intO \xi^2 \bigabs{ \difp \nabla \bue }^2 +\xi^2
    \bigabs{ \nabla \difp \bue }^2 +
    \bigabs{ \difp \F (\bD\bue) }^2
    + \phi\big(\xi \abs{\nabla \difp\bu}\big) + \phi\big(\xi
      \abs{ \difp \nabla \bu}\big)\, d\bx     
    \\
    &\qquad \leq   c \intO\xi ^2 \difp{\bS^\epsilon(\bD\bue)}\cdot
    \difp \bD\bue\,d\bx + c( \norm{\xi}_{1,\infty}, 
      \norm{a}_{C^{1,1}})\hspace*{-4mm} \int_{\Omega\cap \spt
        \xi}\hspace*{-4mm} \phi \big (\abs{ \nabla\bue }\big ) \,
      d\bx\,.
  \end{aligned}
\end{equation*}
The terms $I_1$--$I_7$ are estimated exactly as
in~\cite[(3.17)--(3.22)]{br-reg-shearthin}, while $I_8$ is estimated as
the term $I_{15}$ in~\cite[(4.20)]{br-reg-shearthin}. Thus, we get 
\begin{align*}
  & \intO \vep\, \xi^2 \bigabs{ \difp \nabla \bue }^2 \!+\!\vep \,\xi^2\bigabs{\nabla
    \difp\bue }^2 \! +\!\xi^2 \bigabs{ \difp \F (\bD\bue) }^2 \!+\! \phi (\xi
    \abs{\difp\nabla \bue}) \!+ \!\phi (\xi \abs{\nabla \difp \bue})\, d\bx
    \notag
  \\
  &\le c(\|\ff\|_{p'}, \norm{\xi}_{2,\infty},\norm{a}_{C^{2,1}},
    \delta )\,. %  \label{eq:tang1}
\end{align*}
This proves the second estimate in~\eqref{est-eps} by standard
arguments. The first estimate in~\eqref{est-eps} is proved in the same
way with many simplifications, since we work in the interior where the
method works for all directions. This estimate implies that $\bue \in
W^{2,2}_{\loc }(\Omega)$ and that the system~\eqref{eq-e} is
well-defined point-wise a.e.~in $\Omega$. 

%
% \begin{remark}
% \label{rem:estimate_not_epsilon}
%   If we assume a little more on the external force, say $f\in
%   L^{p'}(\Omega)$, then we did not need to absorb terms arising from
%   the external force in the part concerning the quadratic term,
%   obtaining then the following improved estimate, which will be used
%   later on
% \begin{equation}
%   \epsilon        \intO\xi^2 \abs{\td \bD\bue}^2\,d\bx+\intO\xi^2
%   \abs{\td \F(\bD\bue)}^2\,d\bx 
%   % \\
%   % &\hspace{5cm}
%   \le c (\norm{\ff}_{p'},
%   \norm{\xi}_{2,\infty},\norm{a}_{C^{2,1}},\delta^{-1}) \,.
% \end{equation}
% \end{remark}
% \begin{remark}
% \label{rem:interior}
%   Clearly in the interior the same argument can be applied to all
%   partial derivatives, proving that all second order derivatives
%   of $\bue$ belong to $L^{2}_{loc}(\Omega)$.
% \end{remark}
%
To estimate the derivatives in the $x_{3}$ direction we use
equation~\eqref{eq-e} and it is at this point that we have changes
with respect to the results in~\cite{br-reg-shearthin}. In fact, as
usual in elliptic problems, we have to recover the partial derivatives
with respect to $x_{3}$ by using the information on the tangential
ones. In this problem the main difficulty is that the leading order
term is nonlinear and depends on the symmetric part of the
gradient. Thus, we have to exploit the properties of
$(p,\delta)$-structure of the tensor $\bS$
(cf.~Definition~\ref{ass:1}).  Denoting, for\footnote{Recall that we
  use the summation convention over repeated Greek lower-case letters
  from $1$ to $2$. } $i=1,2,3$,\linebreak 
%$A_{\alpha \gamma}:=\partial_{\gamma 3}S^{\epsilon}_{\alpha
% 3}(\bD\bue)$, $\mathfrak {b}_\gamma:=\partial_3 D_{\gamma 3}\bue$,
%and
$\mathfrak {f}_i :=-f_i -\partial_{\gamma \sigma}S_{i
  3}(\bD\bue)\partial_3 D_{\gamma \sigma}\bue-
\sum_{k,l=1}^3\partial_{k l}S_{i \beta}(\bD\bue)\partial_\beta
D_{k l}\bue$, we can re-write the equations in~\eqref{eq-e}
as follows
  \begin{equation*}
    \label{eq:linear_system}
    \sum_{k=1}^3\partial_{k 3}S_{i 3}(\bD\bue)\partial_3
    D_{k 3}\bue +\partial_{3\alpha}S_{i 3}(\bD\bue)\partial_3
    D_{ 3\alpha}\bue  =\mathfrak{f}_i\qquad\textrm
    {a.e.\ in }\Omega\,.
  \end{equation*}
  Contrary to the corresponding
  equality~\cite[equation~(3.49)]{br-reg-shearthin}, here we use
  directly all the equations in~\eqref{eq}, and not only the first
  two. Now we multiply these equations not by $\partial _3 D_{i3}\bue$
  as expected, but by $\partial _3 \widehat D_{i 3}\bue$, where
  $\widehat D_{\alpha\beta}\bue =0$, for $\alpha,\beta=1,2$,
  $\widehat D_{\alpha 3}\bue =\widehat D_{3\alpha}\bue
  =2D_{\alpha3}\bue$, for $\alpha=1,2$,
  $\widehat D_{33}\bue =D_{33}\bue $.  Summing over $i=1,2,3$ we get,
  by using the symmetries in Remark~\ref{rem:delta0} (iii), that
  \begin{align}
    \label{eq:linear1}
    \begin{split}
      &4\, \partial _{\alpha3} S^\vep_{\beta3}(\bD\bue) \partial
      _3D_{\alpha3} \bue\partial _3D_{\beta3} \bue + 2\, \partial
      _{\alpha 3} S^\vep_{33}(\bD\bue) \partial _3D_{\alpha3} \bue\partial
      _3D_{33} \bue
      \\
      &\quad  +2\, \partial _{33} S^\vep_{\beta3}(\bD\bue) \partial _3D_{33}
      \bue\partial _3D_{\beta3} \bue + \partial _{33}
      S^\vep_{33}(\bD\bue) \partial _3D_{33} \bue\partial _3D_{33} \bue
      \\
      &=\sum_{i=1}^3 \mathfrak{f}_i \,\partial _3 \widehat D_{i3}\bue
      \qquad\qquad\textrm {a.e.\ in }\Omega\,.
    \end{split}
  \end{align}
  To obtain a lower bound for the left-hand side we observe that the terms on the left-hand
  side of~\eqref{eq:linear1} containing $\bS$ are equal to
  \begin{equation*}
    \sum\limits_{i,j,k,l=1}^3\hspace*{-3mm} \partial_{kl} S_{ij}
    (\bD\bue ) Q_{ij}Q_{kl},
  \end{equation*}
  if we choose
  $\bQ =\partial _3 \overline{\bD}\bue$, where $\overline{D}_{\alpha\beta}\bue =0$,
 for $\alpha,\beta=1,2$, $\overline{D}_{\alpha 3}\bue =\overline{D}_{3\alpha}\bue
 =D_{\alpha3}\bue$, for $\alpha=1,2$, and $\overline{D}_{33}\bue =D_{33}\bue $. Thus it follows
 from the coercivity estimate in~\eqref{eq:ass_S} that these terms are bounded from below by
 $\kappa_0\phi''(\abs{\bD\bue}) \abs{\partial _3 \overline{\bD} \bue}^2$. Similarly we see
 that the remaining terms on the left-hand side of~\eqref{eq:linear1} are equal to $\vep
 \abs{\partial _3 \overline{\bD}\bue}^2$. Denoting $\mathfrak b_i :=\partial_3 D_{i3}\bue$,
 $i=1,2,3$, we see that $\abs{\mathfrak b} \sim \abs{\widehat \bD\bue} \sim
 \abs{\overline{\bD}\bue}$. Consequently, we get from~\eqref{eq:linear1} the estimate
  \begin{equation*}
    \left( \epsilon+\phi''(|\bD\bu_\vep|)\right)| {\boldsymbol {
        \mathfrak b}}| \leq |\boldsymbol { \mathfrak f}|\qquad\textrm {a.e.\ in }\Omega\,.   
  \end{equation*}
  By straightforward manipulations (cf.~\cite[Sections 3.2 and
  4.2]{br-reg-shearthin}) we  can estimate the right-hand side as follows
  \begin{equation*}
    \begin{aligned}
      |\boldsymbol { \mathfrak f}| \leq c \left(|\bff|
        +(\epsilon+\phi''(|\bD\bu_\vep|))\left(|\partial_\tau\nabla
          \bu_\vep|+\|\nabla a\|_{\infty}|\nabla^2 \bu_\vep|\right)\right).
    \end{aligned}
  \end{equation*}
  Note that we can deduce from $\boldsymbol { \mathfrak b}$
  information about
  $\widetilde {\mathfrak b}_i:=\partial ^2_{33}u^i_\vep$, $i=1,2,3$,
  because
  $ |\boldsymbol { \mathfrak b}|\geq2|\widetilde{\boldsymbol{\mathfrak
      b}}|-|\partial_{\tau}\nabla\bu_\vep|-\|\nabla
  a\|_{\infty}|\nabla ^{2}\bu_\vep|$ holds a.e.~in $\Omega_P$. This and the last last two
  inequalities imply a.e.~in $\Omega_P$
  \begin{align*}
    &\left( \epsilon+\phi''(|\bD\bu_\vep|)\right)| {\widetilde {\boldsymbol {
        \mathfrak b}}}| 
        \leq  c \left(|\bff|+(\epsilon+\phi''(|\bD\bu_\vep|))\left(|\partial_\tau\nabla
          \bu_\vep|+\|\nabla a\|_{\infty}|\nabla^2 \bu_\vep|\right)\right).
  \end{align*}
  Adding on both sides, for
  $\alpha=1,2$ and $i,k=1,2,3$ the term
  \begin{equation*}
    \left(\vep+\phi''(|\bD\bu_\vep|)\right)\,|\partial_\alpha\partial_i
    u^k_\vep|\,,
  \end{equation*} 
  and using on the right-hand side the definition of the tangential
  derivative (cf.~\eqref{eq:1}), we finally arrive at
  \begin{equation*}
    \begin{aligned}
      &\left( \epsilon+\phi''(|\bD\bu_\vep|)\right)| {\nabla
        ^2\bu_\vep }| \leq c
      \left(|\bff|+(\epsilon+\phi''(|\bD\bu_\vep|))\left(|\partial_\tau\nabla
          \bu_\vep|+\|\nabla a\|_{\infty}|\nabla^2
          \bu_\vep|\right)\right)\,,
    \end{aligned}
  \end{equation*}
  which is valid a.e.~in $\Omega_P$. Note that the constant $c$ only
  depends on the characteristics of $\bS$.  Next, we can choose the
  open sets $\Omega_{P}$ in such a way that
  $\|\nabla a_{P}(x)\|_{{\infty},\Omega_{P}}$ is small enough, so that
  we can absorb the last term from the right hand side, which yields
  \begin{equation*}
    \label{eq:estimate_normal_final}
    \begin{aligned}
      & \left( \epsilon+\phi''(|\bD\bu_\vep|)\right)|\nabla ^{2}\bu_\vep|
      %\\
       %&
       %\qquad 
       \leq c\left(|\bff|+\big(\epsilon+
         \phi''(|\bD\bu_\vep|)\big)|\partial_{\tau}\nabla\bu_\vep|\right)
       \textrm{ a.e. in }\Omega_P\,,
    \end{aligned}
  \end{equation*}
  where again the constant $c$ only depends on the characteristics of
  $\bS$. 
%  We use that the stress
%  tensor which is the sum of the quadratic part with one with
%  $(p,\delta)$-structure for $p<2$.  Hence, 
  By neglecting the second term on the left-hand side (which is
  non-negative), raising the remaining inequality to the power $2$,
  and using that $\bS$ has $(p,\delta)$-structure for $p<2$ we obtain
  \begin{equation*}
    \label{eq:6}
    \epsilon   \int_\Omega \xi_P^2|\nabla^2\bu_\vep|^2\,d\bx\leq
    c\int_\Omega|\bff|^2\,d\bx+\frac{(\epsilon+\delta^{2(p-2)})}{\epsilon}\
    \left(    \epsilon\int_\Omega      \xi_P^2 |\partial_\tau\nabla\bu_\vep|^2\,d\bx\right)
  \end{equation*}
  The already proven results on tangential derivatives and Korn's
  inequality imply that the last integral from right-hand side is
  finite. Thus, the properties of the covering imply the last estimate
  in~\eqref{est-eps}. 
\end{proof}
\subsection{Improved estimates for normal derivatives}
%
%We start by observing that the analysis of tangential derivatives as
%in Proposition~\ref{prop:JMAA2017-2} shows for the perturbed problem
%the following result
%
The proof of~\eqref{est-eps-1} used the system \eqref{eq-e} and
resulted in an estimate that is not uniform with respect to $\eps$.
In this section, by following the ideas in~\cite{SS00}, we proceed
differently and estimate $\mathcal P_3$ in terms of quantities
occurring in~\eqref{est-eps}. 
% we
% prove the following crucial results, obtained working directly with
% the needed quantity and with partial derivatives (not with
% approximation by finite differences). This is made possible by the
% results of Proposition~\ref{prop:JMAA2017-1}. 
The main technical step of the paper is the proof of the following
result:
\begin{Proposition}
  \label{prop:main}
  Let the same hypotheses as in Theorem~\ref{thm:MT} be satisfied with
  $\delta >0$ and
  let the local description $a_P$ of the boundary and the localization
  function $\xi_P$ satisfy $(b1)$--\,$(b3)$ and $(\ell 1)$
  (cf.~Section~\ref{sec:bdr}). Then, there exist a constant $C_2>0$
  such that the weak solution
  $\bue\in W^{1,2}_{0}(\O)$ of the approximate problem~\eqref{eq-e}
  satisfies  for every $P\in \partial \Omega$
  \begin{equation*}
   \eps \int_\Omega \xi^2_P |\partial_3\bD\bue|^2\,d\bx +   \int_\Omega \xi^2_P |\partial_3\bF(\bD\bue)|^2\,d\bx\leq C
    (\norm{\ff}_{p'},\norm{\xi_P}_{2,\infty},\norm{a_P}_{C^{2,1}},\delta,
    C_2)\,,  
  \end{equation*}
  provided $r_P<C_2$ in $(b3)$. 
\end{Proposition}
\begin{proof}%[Proof of Proposition~\ref{prop:main}]
  Let us fix an arbitrary point $P\in \partial \Omega$ and a local
  description $a=a_P$ of the boundary and the localization function
  $\xi=\xi_P$ satisfying $(b1)$--\,$(b3)$ and $(\ell 1)$. In the
  following we denote by $C$ constants that depend only on the
  characteristics of $\bS$.  First we observe that, by the results of
  Proposition~\ref{prop:equivalence} there exists a constant $C_0$,
  depending only on the characteristics of $\bS$, such that
  \begin{equation*}
   \frac{1}{C_0}| \partial_3\bF(\bD\bue)|^2\leq
   \mathcal{P}_3(\bD\bue)\qquad \text{a.e.  in }\Omega.
  \end{equation*}
  Thus, we get, using also the symmetry of $\bD\bue$ and $\bS$,
  \begin{equation*}
    \begin{aligned}
     &\epsilon \sum_{j=1}^3\int_\Omega\xi^2|\partial_3 \bD\bue|^2\,d\bx+
     \frac{1}{C_0} \int_\Omega \xi^2 |\partial_3\bF(\bD\bue)|^2\,d\bx
     \\
     &\leq    \int_\Omega\xi^2 \big (\epsilon \,\partial_3
     \bD\bue +\partial_3 \bS( \bD\bue) \big )\cdot \partial_3
     \bD\bue \,d\bx
      \\
      &=\int_\Omega\sum_{i,j=1}^3 \xi^2 \big (\eps \,\partial_3D_{ij}\bue
      +  \partial_3 S_{ij} (\bD\bue)\big )\partial _3 \partial _ju^i \,d\bx
      \\
      &=\int_\Omega\xi^2 \big (\eps \,\partial_3 D_{\alpha\beta}\bue
      +\partial_3 S_{\alpha\beta} (\bD\bue) \big) \partial_3
      D_{\alpha\beta}\bue\,d\bx
      \\
     &\quad + \int_\Omega\xi^2\big(\epsilon \,\partial_3D_{3\alpha}\bue+ \partial_3S_{3\alpha
        }(\bD\bue)\big)\partial _\alpha D_{33}\bue\,d\bx 
      \\
      &\quad +\int_\Omega \sum_{j=1}^3 \xi^2\partial_3\big ( \vep\,
      D_{j3}\bue +S_{j3}(\bD\bue)\big ) \partial_3^2 u^j_\eps\,d\bx
      \\
       &=:I_{1}+I_{2}+I_{3}\,.
    \end{aligned}
  \end{equation*}
\allowdisplaybreaks
  To estimate $I_2$ we multiply and divide by the quantity
  $\sqrt{\phi''(|\bD\bue|)}\not=0$, use Young's inequality and 
  Proposition \ref{prop:equivalence}. This  yields that for all
  $\param>0$ there exists $c_\param>0$ such that 
  \begin{align*}
%    \begin{aligned}
      |I_2|&\leq \sum_{\alpha
        =1}^2\int_\Omega\xi^2|\partial_3\bS(\bD\bue)|
      |\partial_\alpha\bD\bue|
      \frac{\sqrt{\phi''(|\bD\bue|)}}{\sqrt{\phi''(|\bD\bue|)}}\,d\bx
      \\
      &\quad + \param
      \,\epsilon\int_\Omega\xi^2|\partial_3\bD\bue|^2\,d\bx + c_\param
      \, \epsilon
      \sum_{\alpha=1}^2\int_\Omega\xi^2|\partial_\alpha\bD\bue|^2\,d\bx
      \\
      &\leq \param\int_\Omega\xi^2
      \frac{|\partial_3\bS(\bD\bue)|^2}{\phi''(|\bD\bue|)}\,d\bx+c_\param\sum_{\alpha=1}^2\int_\Omega\xi^2
      \phi''(|\bD\bue|) |\partial_\alpha\bD\bue|^2\,d\bx.
      \\
      &\quad + \param
      \,\epsilon\int_\Omega\xi^2|\partial_3\bD\bue|^2\,d\bx + c_\param
      \, \epsilon
      \sum_{\alpha=1}^2\int_\Omega\xi^2|\partial_\alpha\bD\bue|^2\,d\bx
      \\
      &\leq C \param\int_\Omega\xi^2
      |\partial_3\bF(\bD\bue)|^2\,d\bx+c_\param\sum_{a=1}^2\int_\Omega\xi^2
      |\partial_\alpha\bF(\bD\bue)|^2\,d\bx
      \\
      &\quad + \param
      \,\epsilon\int_\Omega\xi^2|\partial_3\bD\bue|^2\,d\bx + c_\param
      \, \epsilon
      \sum_{\alpha=1}^2\int_\Omega\xi^2|\partial_\alpha\bD\bue|^2\,d\bx\,.
%    \end{aligned}
  \end{align*}
  Here and in the following we denote by $c_\param$ constants that may
  depend on the characteristics of $\bS$ and on $\param^{-1}$, while $C$
  denotes constants that may depend on the characteristics of $\bS$ only.

  To treat the third integral $I_3$ we proceed as follows: We use the
  well-known algebraic identity, valid for smooth enough vectors
  $\bv $ and $i,j,k=1,2,3$,
  \begin{align}
    \label{eq:ai}
    \partial_j\partial_kv^i= \partial _j D_{ik}\bv +\partial _k D_{ij}\bv -\partial _i D_{jk}\bv \,,
  \end{align}
  and the equations~\eqref{eq-e} point-wise, which can be written for
  $j=1,2,3$ as,
  \begin{equation*}
    \partial_3\big(\epsilon\, D_{j3}\bue+
    S_{j3}(\bD\bue)\big)=-f^j-\partial_\beta\big (\epsilon\,D_{j\beta}\bue+
    S_{j\beta}(\bD\bue)\big)\qquad \text{a.e. in     }\Omega\,.
  \end{equation*}
  This is possible due to Proposition \ref{prop:JMAA2017-1}.  Hence,
  we obtain
  \begin{equation*}
        |I_3|\leq  \sum _{j=1}^3\left|\int_\Omega\xi^2\big(-f^j
      -\partial_\beta S_{j\beta}(\bD\bue)-\epsilon\partial_\beta
      D_{j\beta}\bue\big)\big (2\partial _3D_{j3}\bue - \partial _j D_{33}\bue\big )\,d\bx\right|\,.
  \end{equation*}
  The right-hand side can be estimated similarly as $I_2$.  This  yields that for all
  $\param>0$ there exists $c_\param>0$ such that 
  estimated by 
  \begin{equation*}
  \begin{aligned}
    |I_3|&\le \int_\Omega\xi^2\big(|\bff|+\sum_{\beta=1}^2
    |\partial_\beta\bS(\bD\bue)|\big) \big (2|\partial _3\bD\bue| +
    \sum_{\alpha=1}^2|\partial _\alpha\bD\bue|\big )
    \frac{\sqrt{\phi''(|\bD\bue|)}}{\sqrt{\phi''(|\bD\bue|)}}\,d\bx
    \\
    &\quad +\param\,
    \epsilon\int_\Omega\xi^2|\partial_3\bD\bue|^2\,d\bx
    +c_\param\,\epsilon
    \sum_{\beta=1}^2\int_\Omega\xi^2|\partial_\beta\bD\bue|^2\,d\bx
    \\
%     &\leq \param \int_\Omega\xi^2\phi''(|\bD\bue|)|\partial
%     _3\bD\bue|^2\,d\bx +c_\param\sum_{\beta=1}^2\int_\Omega\xi^2
%     \frac{|\partial_\beta\bS(\bD\bue)|^2}{\phi''(|\bD\bue|)}\,d\bx
%     \\
%     &\quad
%     +\frac{\epsilon}{8}\int_\Omega\xi^2|\partial_3\bD\bue|^2\,d\bx
% %\frac{\epsilon}{2}\int_\Omega\xi^2|\partial_3\bD\bue|^2\,d\bx
%     +2\,\epsilon \sum_{\beta=1}^2\int_\Omega\xi^2|\partial_\beta\bD\bue|^2\,d\bx
%     \\
    & \leq\param\, C \!\int_\Omega\xi^2|\partial_3\bF(\bD\bue)|^2\,d\bx
    +c_\param\sum_{\beta=1}^2\int_\Omega\xi^2|\partial_\beta\bF(\bD\bue)|^2\,d\bx
    +\param \,\epsilon\!\int_\Omega\xi^2|\partial_3\bD\bue|^2\,d\bx 
   \\
    &\quad
    +c_\param \,\epsilon \sum_{\beta=1}^2\int_\Omega\xi^2|\partial_\beta\bD\bue|^2\,d\bx 
    + c_\param\int_\Omega
    \frac{\xi^2|\bff|^2}{\phi''(|\bD\bue|)}\,d\bx
    \\
    & \leq\param\, C \!\int_\Omega\xi^2|\partial_3\bF(\bD\bue)|^2\,d\bx
    +c_\param\sum_{\beta=1}^2\int_\Omega\xi^2|\partial_\beta\bF(\bD\bue)|^2\,d\bx
    +\param \,\epsilon\! \int_\Omega\xi^2|\partial_3\bD\bue|^2\,d\bx 
   \\
    &\quad
    +c_\param \,\epsilon \sum_{\beta=1}^2\int_\Omega\xi^2|\partial_\beta\bD\bue|^2\,d\bx 
    +c_\param\big(\|\bff\|_{p'}^{p'}+\|\bD\bue\|_p^p +\delta^p\,\big)\,. 
  \end{aligned}
\end{equation*}
%where $\rho_\phi(g):=\int_\Omega \phi(\abs{g(\bx)})\, d\bx$. 
Observe that we used $p\leq2$ to
estimate the term involving $\bff$.

To estimate $I_{1}$ we employ the algebraic
identity~\eqref{eq:ai} to split the integral as follows
\begin{equation*}
  \begin{aligned}
    I_{1}&=\int_\Omega\xi^2 \big ( \eps \,\partial_3D_{\alpha\beta}\bue
    +\partial_3S_{\alpha\beta}(\bD\bue)\big )
    \big(\partial_\alpha D_{3\beta}
    \bue+ \partial_\beta D_{3\alpha}\bue\big)\,d\bx
    \\
    &\quad -\int_\Omega\xi^2\big (\eps \,\partial_3D_{\alpha\beta}\bue
    + \partial_3S_{\alpha\beta}(\bD\bue)\big )\partial_{\beta}\partial_{\alpha}u_\vep^3\,d\bx 
    \\
    &=:A+B\,.
  \end{aligned}
\end{equation*}
The first term is estimated similarly as $I_2$, yielding  for all $\param>0$ 
\begin{equation*}
  \begin{aligned}
    |A|% &\leq\int_\Omega\xi^2|\partial_3S_{\alpha\beta}(\bD\bue)|
    % \big( |\partial_\alpha\bD\bue|+|\partial_\beta\bD\bue|\big)
    % \frac{\sqrt{\phi''(|\bD\bue|)}}{\sqrt{\phi''(|\bD\bue|)}}\,d\bx
    % \\
    % &\quad + \frac {\epsilon}{4}\int_\Omega\xi^2|\partial_3\bD\bue|^2\,d\bx 
    % + 8\,\epsilon
    % \sum_{\beta=1}^2\int_\Omega\xi^2|\partial_\beta\bD\bue|^2\,d\bx 
    % \\
    % &\leq \param\int_\Omega\xi^2
    % \frac{|\partial_3\bS(\bD\bue)|^2}{\phi''(|\bD\bue|)}\,d\bx+c_\param\sum_{\beta=1}^2\int_\Omega\xi^2
    % \phi''(|\bD\bue|) |\partial_\beta\bD\bue|^2\,d\bx.
    % \\
    % &\quad + \frac {\epsilon}{4}\int_\Omega\xi^2|\partial_3\bD\bue|^2\,d\bx 
    % + 8\,\epsilon
    % \sum_{\beta=1}^2\int_\Omega\xi^2|\partial_\beta\bD\bue|^2\,d\bx 
    % \\
    &\leq C \param\int_\Omega\xi^2
    |\partial_3\bF(\bD\bue)|^2\,d\bx+c_\param\sum_{\beta=1}^2\int_\Omega\xi^2
    |\partial_\beta\bF(\bD\bue)|^2\,d\bx
    \\
    &\quad + \param \,\epsilon\int_\Omega\xi^2|\partial_3\bD\bue|^2\,d\bx 
    + c_\param\,\epsilon
    \sum_{\beta=1}^2\int_\Omega\xi^2|\partial_\beta\bD\bue|^2\,d\bx \,.
  \end{aligned}
\end{equation*}
To estimate $B$ we observe that by the definition of the tangential
derivative we have
\begin{equation*}
  \partial_{\alpha}\partial_{\beta}  u_\vep^{3}=\partial_{\alpha}\partial_{\tau_{\beta}}u_\vep^{3}- 
  (\partial_{\alpha}\partial_{\beta}\,a)\ D_{33}\bue-(\partial_{\beta }\,a)\
  \partial_{\alpha} D_{33}\bue,
\end{equation*}
and consequently the term $B$ can be split into the following three
terms: 
\begin{equation*}
  \begin{aligned}
    &-\!\int_\Omega\!\xi^2 \big ( \eps \,\partial_3D_{\alpha\beta}\bue
    +\partial_3S_{\alpha\beta}(\bD\bue)\big )\!
    \left(\partial_{\alpha}\partial_{\tau_{\beta}}u_\vep^{3}- 
      (\partial_{\alpha}\partial_{\beta}a) D_{33}\bue-(\partial_{\beta }a)
      \partial_{\alpha} D_{33}\bue\right)\,d\bx
    \\
    &=:B_{1}+B_{2}+B_{3}\,.
  \end{aligned}
\end{equation*}
We estimate $B_{2}$ as follows
\begin{equation*}
  \begin{aligned}
    | B_{2}|&\leq
    \int_\Omega\xi^2|\partial_3\bS(\bD\bue)| 
    |\nabla^{2}a| |\bD\bue|
    \frac{\sqrt{\phi''(|\bD\bue|)}}{\sqrt{\phi''(|\bD\bue|)}} + \eps \,\xi^2
    \,|\partial_3\bD\bue|     |\nabla^{2}a| |\bD\bue|\,d\bx
    \\
    &\leq \param \int_\Omega\xi^2
    \frac{{|\partial_3\bS(\bD\bue)|^2}}{{\phi''(|\bD\bue|)}}\,d\bx
    +c_{\param}\|\nabla^{2}a\|_{\infty}^2\int_\Omega\xi^2
    |\bD\bue|^{2} \phi''(|\bD\bue|)\,d\bx
    \\
    &\quad + \param \,\epsilon\int_\Omega\xi^2|\partial_3\bD\bue|^2\,d\bx 
    +c_\param\, \epsilon \,\|\nabla^{2}a\|_{\infty}^2\int_\Omega\xi^2|\bD\bue|^2\,d\bx 
    \\
    &\leq \param C\int_\Omega\xi^2|\partial_{3}\bF(\bD\bue)|^{2}\,d\bx
    +c_{\param}\|\nabla^{2}a\|_{\infty}^2\rho_{\phi}(|\bD\bue|)
    \\
    &\quad + \frac {\epsilon}{8}\int_\Omega\xi^2|\partial_3\bD\bue|^2\,d\bx 
    + 2\,\epsilon \,\|\nabla^{2}a\|_{\infty}^2\|\bD\bue\|_2^2 \,.
\end{aligned}
\end{equation*}
The term  $B_{3}$ is estimated similarly as $I_2$, yielding  for all
$\param>0$  
\begin{equation*}
  \begin{aligned}
   | B_{3}|% &\leq \sum_{\beta=1}^2\int_\Omega\xi^2|\partial_3\bS(\bD\bue)|
    %  | \nabla a| |\partial_{\beta}\bD\bue|
    %  \frac{\sqrt{\phi''(|\bD\bue|)}}{\sqrt{\phi''(|\bD\bue|)}}\,d\bx
    %  \\
    %  &\leq \param \int_\Omega\xi^2
    %  \frac{{|\partial_3\bS(\bD\bue)|^2}}{{\phi''(|\bD\bue|)}}\,d\bx
    %  +c_{\param}\|\nabla a\|_{\infty}^2\sum_{\beta=1}^2
    %  \int_\Omega\xi^2 |\partial _\beta\bD\bue|^{2}
    %  \phi''(|\bD\bue|)\,d\bx   
    % \\
    &\leq \param C
    \int_\Omega\xi^2|\partial_{3}\bF(\bD\bue)|^{2}\,d\bx+c_{\param}\|\nabla
    a\|_{\infty}^2\sum_{\beta=1}^2     \int_\Omega\xi^2
    |\partial_{\beta}\bF(\bD\bue)|^{2}\,d\bx
    \\
    &\quad + \param\,\epsilon\int_\Omega\xi^2|\partial_3\bD\bue|^2\,d\bx 
    + c_\param\,\epsilon \,\|\nabla
    a\|_{\infty}^2\sum_{\beta=1}^2  
    \sum_{\beta=1}^2\int_\Omega\xi^2|\partial_\beta\bD\bue|^2\,d\bx \,.
%    \\
%    &\leq c\param
%    \int_\Omega\xi^2|\partial_{3}\bF(\bD\bue)|^{2}\,d\bx+c_{\param}\|\nabla 
%    a\|_{\infty}     \int_\Omega\xi^2
%    |\partial_{\tau_{\alpha}}\bF(\bD\bue)|^{2}\,d\bx
%    \\
%    &\qquad\qquad+c_{\param}\|\nabla a\|_{\infty} ^2    \int_\Omega\xi^2
%    |\partial_{3}\bF(\bD\bue)|^{2}\,d\bx.
\end{aligned}
\end{equation*}
% Concerning the term $B_1$, we need to perform some integration by parts, which can be
% justified by means of the regularity results valid for the solutions of our approximate
% problem~\eqref{eq-e}, this is one of the relevant points which we are adapting
% from~\cite{SS2000}.  In particular, one would use the fact that for smooth enough $\bue$
Concerning the term $B_1$, we would like to perform some integration
by parts, which is one of the crucial observations we are adapting
from~\cite{SS00}. Neglecting the localization  $\xi$ in $B_1$ we would
like to use that 
\begin{equation}
  \label{eq:trick_integrate}
  \int_\Omega \partial_3S^\eps_{\alpha\beta}(\bD\bue)
  \partial_{\alpha}\partial_{\tau_{\beta}}u_\vep^{3}\,d\bx=\int_\Omega\partial_\alpha S^\eps_{\alpha\beta}(\bD\bue)
  \partial_{3}\partial_{\tau_{\beta}}u_\vep^{3}\,d\bx\,.
\end{equation}
This formula can be justified by using an appropriate approximation,
that exists for $\bue\in W^{1,2}_0(\Omega)\cap W^{2,2}(\Omega)$ since
$\partial_{\tau}\bue=\bfzero$ on $\partial \Omega$. More precisely, to
treat the term $B_1$ we use that the solution $\bue$ of~\eqref{eq-e}
belongs to $W^{1,2}_0(\Omega)\cap W^{2,2}(\Omega)$. Thus,
$\partial_{\tau}\big ( {\bue} _{|\Omega_P}\big )=\bfzero$ on
$\partial\Omega_P\cap\partial\Omega$, hence
$\xi_P\partial_{\tau}(\bue^3)=\bfzero$ on $\partial\Omega$.  This
implies that we can find a sequence
$(\boldsymbol {\mathcal{S}}_n,\boldsymbol{\mathcal{U}}_n)\in
C^\infty(\Omega)\times C^\infty_0(\Omega)$ such that
$(\boldsymbol{\mathcal{S}}_n,\boldsymbol{\mathcal{U}}_n)\to(\bS^\epsilon,\partial_\tau\bue)$
in $W^{1,2}(\Omega)\times W^{1,2}_0(\Omega)$ and perform calculations
with $(\boldsymbol{\mathcal{S}}_n,\boldsymbol{\mathcal{U}}_n)$,
showing then that all formulas of integration by parts are
valid. Passage to the limit as $n\to+\infty$ is done only in the last
step. For simplicity we drop the details of this well-known argument
(sketched also in~\cite{SS00}) and we write directly formulas without
this smooth approximation.  Thus, performing several integrations by
parts, we get
% In our
% case we have also the localization, hence more terms are present in
% the formulas. Moreover, the function $\bue$ is not smooth enough to
% justify the two integrations by parts needed to obtain the latter
% inequality. (More partial derivatives are needed in the intermediate
% formulas than those appearing in the final
% one~\eqref{eq:trick_integrate}). Nevertheless, we can find a sequence
% $\{(\boldsymbol
% {\mathcal{S}}_n,\boldsymbol{\mathcal{U}}_n)\}_{n\in\setN}$, belonging
% to $C^\infty(\Omega)\times C^\infty_0(\Omega)$ such that
% $(\boldsymbol{\mathcal{S}}_n,\boldsymbol{\mathcal{U}}_n)\to(\bS^\epsilon,\partial_\tau\bue)$
% in $W^{1,2}(\Omega)\times W^{1,2}_0(\Omega)$ and perform calculations
% with $(\boldsymbol{\mathcal{S}}_n,\boldsymbol{\mathcal{U}}_n)$,
% showing then that all formulas of integration by parts are
% valid. Passage to the limit as $n\to+\infty$ is done only in the last
% step. For simplicity we drop the details of this well-known argument
% (sketched also in~\cite{SS00}) and we write directly formulas without
% this smooth approximation.  \marginpar{formulas wrong in your version,
%   now they seem correct. The next formula was nevertheless ok also in
%   your file} Thus, performing several integrations by parts, we get
\begin{equation*}
  \begin{aligned}
     & \int_\Omega\xi^2\partial_3S_{\alpha\beta}(\bD\bue)
    \partial_{\alpha}\partial_{\tau_{\beta}}u_\vep^{3}\,d\bx
    \\
    &=
    \int_\Omega \big(\partial _{\alpha}\xi^2\big) S_{\alpha\beta}(\bD\bue)
    \partial_3 \partial_{\tau_{\beta}}u_\vep^{3}\,d\bx % + \int_\Omega
    -\int_\Omega \big(\partial_{3}\xi^2\big) S_{\alpha\beta}(\bD\bue)
    \partial_\alpha\partial_{\tau_{\beta}}u_\vep^{3}\,d\bx
    \\
    &\quad+\int_\Omega
    \xi^2\partial_\alpha S_{\alpha\beta}(\bD\bue)
    \partial_3\partial_{\tau_{\beta}}u_\vep^{3}\,d\bx
  \end{aligned}
\end{equation*}
and % Similarly we get
\begin{equation*}
  \begin{aligned}
     & \eps \int_\Omega\,\xi^2\partial_3D_{\alpha\beta}\bue
    \partial_{\alpha}\partial_{\tau_{\beta}}u_\vep^{3}\,d\bx
    \\
    &=
    \eps \int_\Omega \big(\partial _{\alpha}\xi^2\big) D_{\alpha\beta}\bue
    \partial_3\partial_{\tau_{\beta}}u_\vep^{3}\,d\bx
- \eps \int_\Omega \big(\partial_{3}\xi^2\big) D_{\alpha\beta}\bue
   \partial_\alpha \partial_{\tau_{\beta}}u_\vep^{3}\,d\bx
% -\int_\Omega \big(\partial_{3}\xi^2\big) D_{\alpha\beta}\bue
%     \partial_\alpha\partial_{\tau_{\beta}}u_\vep^{3}\,d\bx
    \\
    &\quad+\eps \int_\Omega
    \xi^2\partial_\alpha D_{\alpha\beta}\bue
    \partial_3\partial_{\tau_{\beta}}u_\vep^{3}\,d\bx\,.
  \end{aligned}
\end{equation*}
This shows that 
\begin{equation*}
\begin{aligned}
B_1    &=
    \int_\Omega 2\xi \partial_{\alpha}\xi \,S_{\alpha\beta}(\bD\bue)
    \partial_3    \partial_{\tau_{\beta}}u_\vep^{3}\,d\bx
    -\int_\Omega 2\xi\partial_{3}\xi\, S_{\alpha\beta}(\bD\bue)
    \partial_\alpha\partial_{\tau_{\beta}}u_\vep^{3}\,d\bx
    \\
    &\quad+\int_\Omega
    \xi^2\partial_\alpha S_{\alpha\beta}(\bD\bue)
    \partial_3\partial_{\tau_{\beta}}u_\vep^{3}\,d\bx
    + \eps     \int_\Omega 2\xi \partial_{\alpha}\xi \,D_{\alpha\beta}\bue
    \partial_3    \partial_{\tau_{\beta}}u_\vep^{3}\,d\bx
    \\
    &\quad -\eps \int_\Omega 2\xi\partial_{3}\xi\, D_{\alpha\beta}\bue
    \partial_\alpha\partial_{\tau_{\beta}}u_\vep^{3}\,d\bx
    + \eps \int_\Omega
    \xi^2\partial_\alpha D_{\alpha\beta}\bue
    \partial_3\partial_{\tau_{\beta}}u_\vep^{3}\,d\bx
    \\
    &=:B_{1,1}+B_{1,2}+B_{1,3}+B_{1,4}+B_{1,5}+B_{1,6}\,.
  \end{aligned}
\end{equation*}
To estimate $B_{1,1}, B_{1,3}, B_{1,4}, B_{1,6}$ we observe that
\begin{equation*}
\partial_3\partial_{\tau_{\beta}}u_\vep^{3}=\partial_{\tau_{\beta}}\partial_3u_\vep^{3}
%=\bD_{33}(\partial_{\tau_{\beta}}\bue)
=\partial_{\tau_{\beta}}D_{33}\bue\,. 
\end{equation*}
By using Young inequality, the growth properties of $\bS$ in~\eqref{eq:hammere} and
\eqref{eq:tang} we get
\begin{equation*}
  \begin{aligned}
    \left|   B_{1,1}\right|&\leq
%\|\nabla^2\xi\|_{\infty}\|\bS(\bD\bue)\|_{p'}\|\nabla
%    \bue\|_{p}+
 %   \\
  %  &\qquad\qquad
    \|\nabla\xi\|_{\infty}^2\int_\Omega\frac{|\bS(\bD\bue)|^2}{\phi''(|\bD\bue|)}\,d\bx
    +C\sum_{\beta=1}^2\int_\Omega\xi^2 \phi''(|\bD\bue|)|\partial_{\tau_{\beta}}\bD\bue|^2\,d\bx
    \\
    &\leq
    \|\nabla\xi\|^2_{\infty}\rho_\phi(|\bD\bue|)+C\sum_{\beta=1}^2\int_\Omega\xi^2|\partial_{\tau_\beta}
    \bF(\bD\bue)|^2\,d\bx
  \end{aligned}
\end{equation*}
% \comment{we have to show that $|\partial_{\tau_\beta}
%     \bF(\bD\bue)|^2 \sim \phi''(|\bD\bue|)|\partial_{\tau_{\beta}}\bD\bue|^2$}%
and %Next
\begin{equation*}
  \begin{aligned}
    \left| B_{1,3}\right|&\leq \sum_{\beta=1}^2 \int_\Omega
    \xi^2\frac{|\partial_\beta\bS_{\alpha\beta}(\bD\bue)|^2}{\phi''(|\bD\bue|)}\,d\bx+
    \sum_{\beta=1}^2 \int_\Omega
    \xi^2\phi''(|\bD\bue|)|\partial_{\tau_{\beta}}\bD\bue|^2\,d\bx
    \\
    & \leq C\sum_{\beta=1}^2 \int_\Omega\xi^2
    |\partial_{\beta}\bF(\bD\bue)|^{2}
%\,d\bx
    % +\|\nabla    a\|_{\infty} ^2
    % \int_\Omega\xi^2  |\partial_{3}\bF(\bue)|^{2}\,d\bx
  %  +C\int_\Omega
    +\xi^2|\partial_{\tau_{\beta}}\bF(\bD\bue)|^{2}\,d\bx\,.
  \end{aligned}
\end{equation*}
Similarly we get 
\begin{equation*}
  \begin{aligned}
    \left|   B_{1,4}\right|&\leq
%\|\nabla^2\xi\|_{\infty}\|\bS(\bD\bue)\|_{p'}\|\nabla
%    \bue\|_{p}+
 %   \\
  %  &\qquad\qquad
    % \|\nabla\xi\|_{\infty}^2\int_\Omega\frac{|\bS(\bD\bue)|^2}{\phi''(|\bD\bue|)}\,d\bx
    % +C\sum_{\beta=1}^2\int_\Omega\xi^2 \phi''(|\bD\bue|)|\partial_{\tau_{\beta}}\bD\bue|^2\,d\bx
    % \\
    % &\leq
    C\, \eps  \|\nabla\xi\|^2_{\infty}\|\bD\bue\|_2^2+  C\, \eps \sum_{\beta=1}^2\int_\Omega\xi^2|\partial_{\tau_\beta}
    \bD\bue|^2\,d\bx
  \end{aligned}
\end{equation*}
% \comment{we have to show that $|\partial_{\tau_\beta}
%     \bF(\bD\bue)|^2 \sim \phi''(|\bD\bue|)|\partial_{\tau_{\beta}}\bD\bue|^2$}%
and %Next
\begin{equation*}
  \begin{aligned}
    \left| B_{1,6}\right|% &\leq \sum_{\beta=1}^2 \int_\Omega
    % \xi^2\frac{|\partial_\beta\bS_{\alpha\beta}(\bD\bue)|^2}{\phi''(|\bD\bue|)}\,d\bx+
    % \sum_{\beta=1}^2 \int_\Omega
    % \xi^2\phi''(|\bD\bue|)|\partial_{\tau_{\beta}}\bD\bue|^2\,d\bx
    % \\
    & \leq C\,\eps  \sum_{\beta=1}^2 \int_\Omega\xi^2
    |\partial_{\beta}\bD\bue|^{2}
%\,d\bx
    % +\|\nabla    a\|_{\infty} ^2
    % \int_\Omega\xi^2  |\partial_{3}\bF(\bue)|^{2}\,d\bx
  %  +C\int_\Omega
    +\xi^2|\partial_{\tau_{\beta}}\bD\bue|^{2}\,d\bx\,.
  \end{aligned}
\end{equation*}

To estimate $B_{1,2}$ and $B_{15}$ we observe that, using the
algebraic identity~\eqref{eq:ai} and
the defintion of the tangential derivative,  
\begin{equation*}
  \begin{aligned}
    \partial_\alpha \partial_{\tau_{\beta}}u_\vep^{3}&=\partial_\alpha(\partial_\beta
    \bue^3+\partial_\beta a\ \partial_3u_\vep^3)
    \\
    &=\partial_{\alpha}\partial _{\beta}u_\vep^3+\partial_{\alpha}\partial_{\beta} a\
    D_{33}\bue+\partial_\beta a\ \partial_\alpha    D_{33}\bue
    \\
    &=\partial_\alpha D_{\beta3}\bue+\partial_\beta
    D_{\alpha3}\bue-\partial_3
    D_{\alpha\beta}\bue+\partial_{\alpha}\partial_{\beta} a\ 
    D_{33}\bue+\partial_\beta a\ \partial_\alpha    D_{33}\bue\,.
  \end{aligned}
\end{equation*}
Hence by substituting and again by the same inequalities as before we
arrive to the following estimates
\begin{equation*}
  \begin{aligned}
    |B_{1,2}|&\leq \param C \int_\Omega    \xi^2|\partial_3\bF(\bD\bue)|^2\,d\bx
    +C \big(1+\|\nabla a\|^2_\infty\big )\sum_{\beta=1}^2   \int_\Omega\xi^2
    |\partial_{\beta}\bF(\bD\bue)|^{2} \,d\bx
    \\
    &\quad +    c_\param(1+\|\nabla^2a\|_{\infty}) \|\nabla\xi\|_\infty^2
    \rho_\phi(|\bD\bue|)\,, 
    \\
    |B_{1,5}|&\leq \param\, \eps   \int_\Omega    \xi^2|\partial_3\bD\bue|^2\,d\bx
    +c_\param \big(1+\|\nabla a\|^2_\infty\big )\sum_{\beta=1}^2  \eps  \int_\Omega\xi^2
    |\partial_{\beta}\bD\bue|^{2} \,d\bx
    \\
    &\quad +    c_\param (1+\|\nabla^2a\|_{\infty}) \|\nabla\xi\|_\infty^2
    \eps \|\bD \bue\|_2^2\,. 
  \end{aligned}
\end{equation*}
Collecting all estimates and using that $\norm{\nabla a}_\infty \le r_P \le
1$, we finally obtain
\begin{equation*}
  \begin{aligned}
    \epsilon& %\sum_{j=1}^3
    \int_\Omega\xi^2|\partial_3\bD\bue|^2\,d\bx +\frac
    1{C_0}\int_\Omega\xi^2|\partial_3\bF(\bD\bue)|^2\,d\bx
    \\
    &
    \leq \param\, \epsilon\int_\Omega\xi^2|\partial_3\bD\bue|^2\,d\bx+\param\, 
    C \int_\Omega\xi^2|\partial_3\bF(\bD\bue)|^2\,d\bx
    \\
    & 
    \quad +c_\param %(1+\|\nabla a\|_{\infty}^2)
    \sum_{\beta=1}^2
    \int_\Omega\xi^2 |\partial_{\beta}\bF(\bD\bue)|^{2}
    +\xi^2|\partial_{\tau_{\beta}}\bF(\bue)|^{2}\,d\bx  +c_\param\,\epsilon\sum_{\beta=1}^2\int_\Omega\xi^2|\partial_\beta\bD\bue|^2\,d\bx
     \\
    &\quad+c_\param\big (1+ \norm{\nabla ^2 a}^2_\infty
    +(1+\|\nabla^{2}a\|_{\infty}^2)\|\nabla\xi\|_{\infty}^2\big)\big(\|\bff\|_{p'}^{p'}
  +\rho_\phi(|\bD\bue|)+\rho_\phi(\delta)\big) 
  \\
  &\quad + c_\param \big (1+ \norm{\nabla ^2 a}^2_\infty
    +(1+\|\nabla^{2}a\|_{\infty}^2)\|\nabla\xi\|_{\infty}^2\big)\,\|\bD
    \bue \|_{2}^{2}\,.
    \end{aligned}
\end{equation*}
The quantities that are bounded uniformly in $L^2(\Omega_P)$ are the
tangential derivatives of $\epsilon\,\bD\bue$ and of $\bF(\bD\bue)$. By
definition we have 
\begin{equation*}
  \begin{aligned}
    & \partial_\beta
    \bD\bue=\partial_{\tau_\beta}\bD\bue-\partial_\beta
    a\,\partial_3\bD\bue,
    \\  
    &
    \partial_\beta
    \bF(\bD\bue)=\partial_{\tau_\beta}\bF(\bD\bue)-\partial_\beta
    a\,\partial_3\bF(\bD\bue),
  \end{aligned}
\end{equation*}
and if we substitute we obtain
\begin{equation*}
  \begin{aligned}
    &\epsilon \int_\Omega\xi^2|\partial_3\bD\bue|^2\,d\bx+
    \frac 1{C_0}\int_\Omega\xi^2|\partial_3\bF(\bD\bue)|^2\,d\bx
    \\
    &\leq \epsilon\big(\param\!+\!4\,\|\nabla
    a\|_\infty^2\big)\!\int_\Omega\!\xi^2|\partial_3\bD\bue|^2d\bx
    %\\
    %&
    +\big(\param\,
    C+c_\param\|\nabla a\|_\infty^2\big)\! \int_\Omega\!\xi^2|\partial_3\bF(\bD\bue)|^2d\bx
    \\
    & \quad +c_\param%(1+\|\nabla a\|_{\infty}^2)
    \sum_{\beta=1}^2
    \int_\Omega\xi^2|\partial_{\tau_{\beta}}\bF(\bue)|^{2}\,d\bx 
    +c_\param\,\epsilon\sum_{\beta=1}^2\int_\Omega\xi^2|\partial_{\tau _\beta}\bD\bue|^2\,d\bx 
    \\
    &\quad+c_\param\big (1+ \norm{\nabla ^2 a}^2_\infty
    +(1+\|\nabla^{2}a\|_{\infty}^2)\|\nabla\xi\|_{\infty}^2\big)\big(\|\bff\|_{p'}^{p'} 
    +\rho_\phi(|\bD\bue|)+\rho_\phi(\delta)\big) 
  \\
  &\quad + c_\param\big (1+ \norm{\nabla ^2 a}^2_\infty
    +(1+\|\nabla^{2}a\|_{\infty}^2)\|\nabla\xi\|_{\infty}^2\big)\,\|\bD
    \bue \|_{2}^{2}\,.
    \end{aligned}
\end{equation*}
By choosing first $\param>0$ small enough such that $\lambda\,C<4^{-1}C_0 $ and then
choosing in the local description of the boundary $R=R_P$ small enough such that
$c_\lambda\|\nabla a\|_\infty<4^{-1}C_0 $, we can absorb the first two terms from the
right-hand side into the left-hand side to obtain
\begin{equation*}
  \begin{aligned}
    &\epsilon \int_\Omega\xi^2|\partial_3\bD\bue|^2\,d\bx+
    \frac 1{C_0}\int_\Omega\xi^2|\partial_3\bF(\bD\bue)|^2\,d\bx
    \\
    &\leq c_\param%(1+\|\nabla a\|_{\infty}^2)
    \sum_{\beta=1}^2
    \int_\Omega\xi^2|\partial_{\tau_{\beta}}\bF(\bue)|^{2}\,d\bx 
    +c_\param\,\epsilon\sum_{\beta=1}^2\int_\Omega\xi^2|\partial_{\tau _\beta}\bD\bue|^2\,d\bx 
    \\
    &\quad+c_\param\big (1+ \norm{\nabla ^2 a}^2_\infty
    +(1+\|\nabla^{2}a\|_{\infty}^2)\|\nabla\xi\|_{\infty}^2\big)\big(\|\bff\|_{p'}^{p'} 
    +\rho_\phi(|\bD\bue|)+\rho_\phi(\delta)\big) 
  \\
  &\quad + c_\param\big (1+ \norm{\nabla ^2 a}^2_\infty
    +(1+\|\nabla^{2}a\|_{\infty}^2)\|\nabla\xi\|_{\infty}^2\big)\,\|\bD
    \bue \|_{2}^{2}\,,
    \end{aligned}
\end{equation*}
where now $c_\param$ depends on the fixed paramater $\param$, the
characteristics of $\bS$ and on $C_2$. 
The right-hand side is bounded uniformly with respect to $\epsilon>0$, due to
Proposition~\ref{prop:JMAA2017-1}, proving the assertion of the
proposition. 
% in particular that
% \begin{equation*}
%     \int_{\Omega_P} \xi^2_{P}|\partial_3\bF(\bD\bue)|^2\,d\bx\leq C
%     (\norm{\ff}_{p'},\norm{\xi}_{2,\infty},\norm{a}_{C^{2,1}},\delta,C_2)\,, 
% \end{equation*}
% where $\xi=\xi_{P}$ is a cut-off function with support near to the
% boundary $\partial \Omega$, as defined in Sec.~\ref{sec:bdr}.
%for some constant depending on $P$ and on the data of the problem, but
%not on $\epsilon>0$.
\end{proof}
Choosing now an appropriate finite covering of the boundary (for the
details see also~\cite{br-reg-shearthin}),
Propositions~\ref{prop:JMAA2017-1}-\ref{prop:main} yield the
following result:
\begin{theorem}
\label{thm:estimate_for_ue}
Let the same hypotheses as in Theorem~\ref{thm:MT} with $\delta>0$ be satisfied. Then, it
holds
  \begin{equation*}
    \epsilon\|\nabla\bD\bue\|^2_2  +
    \| \nabla \bF(\bD\bue)\|^2_2\leq
    C(\norm{\ff}_{p'},\delta, \partial \Omega)\,.
  \end{equation*}
  % where $C(\norm{\ff}_{p'},\delta, \partial \Omega)$ is non decreasing in its first two
  % arguments and depends on the boundary through the second order derivatives (in the
  % local
  % representation).
  % for a constant $C$ not depending on $\epsilon>0$
\end{theorem}
\subsection{Passage to the limit}
Once this has been proved, by means of appropriate limiting process we
can show that the estimate is inherited by
$\bu=\lim_{\epsilon\to0}\bue$, since $\bu$ is the unique solution
to the boundary value problem~\eqref{eq}. 
We can now give the proof of the main result
\begin{proof}[Proof (of Theorem~\ref{thm:MT})]
  Let us firstly assume that $\delta >0$. From Proposition
  \ref{thm:existence_perturbation}, Proposition \ref{lem:hammer} and
  Theorem \ref{thm:estimate_for_ue} we know that $\bF(\bD\bue)$ is
  uniformly bounded with respect to $\vep$ in $W^{1,2}(\Omega)$. This
  also implies (cf.~\cite[Lemma 4.4]{bdr-7-5}) that $\bue$ is
  uniformly bounded with respect to $\vep $ in $W^{2,p}(\Omega)$. The
  properties of $\bS$ and Proposition \ref{thm:existence_perturbation}
  also yield that $\bS(\bD\bue)$ is uniformly bounded with respect to
  $\vep $ in $L^{p'}(\Omega)$.  Thus, there exists a subsequence
  $\{\epsilon_n\}$ (which converges to $0$ as $n\to+\infty)$,
  $\bu \in W^{2,p}(\Omega)$, $\widetilde{\bF}\in W^{1,2}(\Omega)$, and
  $\boldsymbol {\chi} \in L^{p'}(\Omega)$ such that
  \begin{equation*}
    \begin{aligned}
      \buen&\rightharpoonup \bu&&\text{in }W^{2,p}(\Omega)\cap
      W^{1,p}_0(\Omega)\,,
      \\
      \bD\buen&\to\bD\bu\quad&&\text{a.e. in }\Omega\,,
      \\
      \bF(\bD\buen) &\rightharpoonup \widetilde{\bF}&&\text{in
      }W^{1,2}(\Omega)\,,
      \\
      \bS(\bD\buen) &\rightharpoonup \boldsymbol \chi&&\text{in
      }L^{p'}(\Omega)\,.
    \end{aligned}
  \end{equation*}
  The continuity of $\bS$ and $\bF$ and the classical result stating that
  the weak limit and the a.e.~limit in Lebesgue spaces coincide
  (cf.~\cite{GGZ}) imply that
  \begin{equation*}
    \begin{aligned}
      \widetilde{\bF} =\bF(\bD\bu) \qquad \text{and}\qquad
      \boldsymbol \chi = \bS(\bD\bu)\,.
    \end{aligned}
  \end{equation*}
  These results enable us to pass to the limit in the weak formulation
  of the perturbed problem~\eqref{eq-e}, which yields
  \begin{equation*}
    \int_\Omega \bS(\bD\bu)\cdot\bD\bv\,d\bx
    =\int_\Omega\bff\cdot\bv\,d\bx\qquad\forall\,\bv \in
    C^{\infty}_0(\Omega)\,,
  \end{equation*}
  where we also used that $\lim _{\vep_n \to 0} \int_\Omega \vep_n
  \bD\buen \cdot \bD\bv\, d\bx =0$. By density we thus know that $\bu$
  is the unique weak solution of problem~\eqref{eq}.
%   where $\bu$ is the unique solution to~\eqref{eq}.  By
%   Theorem~\ref{thm:existence_perturbation} and Theorem~\ref{thm:estimate_for_ue} it
%   follows that there exists a constant $C$, independent of
%   $\epsilon>0$, such that
% \begin{equation*}
%   \int_\Omega|\bF(\bD\bue)|^2+|\nabla\bF(\bD\bue)|^2\,d\bx\leq  C,  
% \end{equation*}
% hence (again up to a sub-sequence) there exists $\widetilde{\bF}\in
% W^{1,2}(\Omega)$ such that
% \begin{equation*}
%   \bF(\bD\buen)\rightharpoonup \widetilde{\bF}\qquad\text{in
%   }W^{1,2}(\Omega).
%     \end{equation*}
% Next, by using the continuity of $\bF$ and the previous results of
% convergence of $\bD\buen$ we have 
%   \begin{equation*}
%     \begin{aligned}
%       &    \bF(\bD\buen)\rightarrow \bF(\bD\bu)\qquad\text{a.e. in }\Omega.
% %      \\
% %      &\bF(\bD\buen)\rightharpoonup \bF(\bD\bu)\qquad\text{in
%  %     }W^{1,2}(\Omega).
%     \end{aligned}
%   \end{equation*}
% We can now identify the limit
% $\widetilde{\bF}$ since in $L^{q}(\Omega)$ (with $q\not=\infty$) weakly and almost
% everywhere sequences converging sequences have the same limit (see for
% instance~\cite{GGZ})
% implying that 
%   \begin{equation*}
%     \begin{aligned}
% %      &    \bF(\bD\buen)\rightarrow \bF(\bD\bu)\qquad\text{a.e. in }\Omega,
%  %     \\
%       &\bF(\bD\buen)\rightharpoonup \bF(\bD\bu)\qquad\text{in
%       }W^{1,2}(\Omega).
%     \end{aligned}
%   \end{equation*}
Finally the lower semi-continuity of the norm implies that
\begin{equation*}
  \int_{\Omega}|\nabla\bF(\bD\bu)|^2\,d\bx\leq\liminf_{\epsilon_n\to0}
  \int_{\Omega}|\nabla\bF(\bD\buen)|^2\,d\bx\leq c, 
\end{equation*}
ending the proof in the case $\delta>0$.

Let us now assume that $\delta=0$. Proposition~\ref{prop:JMAA2017-1} and
Proposition~\ref{prop:main} are valid only for $\delta>0$ and thus cannot be used directly
for the case that $\bS$ has $(p,\delta)$-structure with $\delta=0$. However, it is proved
in~\cite[Section 3.1]{bdr-7-5} that for any stress tensor with $(p,0)$-structure $\bS$,
there exist stress tensors $\bS^\kappa$, having $(p,\kappa)$-structure with $\kappa>0$,
and approximating $\bS$ in an appropriate way. Thus we approximate~\eqref{eq-e} by the
system
\begin{equation*}
%  \label{eq-ek}
  \begin{aligned}
    -\divo \bfS^{\vep,\kappa} (\bfD\bu_{\eps ,\kappa})&=\bff\qquad&&\text{in
    }\Omega\,,
    \\
    \bfu &= \bfzero &&\text{on } \partial \Omega\,,
  \end{aligned}
\end{equation*}
where 
  \begin{equation*}
%    \label{eq:perturbed_Sk}
    \bS^{\epsilon,\kappa}(\bQ):=\epsilon \,\bQ +
    \bS^\kappa(\bQ),\qquad\text{with }\epsilon>0\,,\, \kappa\in (0,1)\,.
  \end{equation*}
  For fixed $\kappa>0$ we can use the above theory and use that fact that the estimates
  are uniformly in $\kappa$ to pass to the limit as $\epsilon\to0$. Thus, we obtain that
  for all $\kappa \in (0,1)$ there exists a unique $\bu_\kappa \in W^{1,p}_0(\Omega)$
  satisfying for all $\bv \in W^{1,p}_0(\Omega)$
  \begin{equation*}%\label{weak}
    \int_\Omega \bS^\kappa(\bD\bu_\kappa)\cdot\bD\bv\,d\bx=\int_\Omega \bff\cdot
    \bv\,d\bx
  \end{equation*}
and 
\begin{equation}
  \label{eq:fin}
  \int_{\Omega}|\bF^\kappa(\bD\bu_\kappa)|^2+|\nabla\bF^\kappa(\bD\bu_\kappa)|^2\,d\bx\leq c
  (\norm{\ff}_{p'}, \partial \Omega) \,, 
\end{equation}
where the constant is independent of $\kappa\in(0,1)$ and $\bF^\kappa\colon\setR^{3 \times
  3} \to \setR^{3 \times 3}_\sym$ is defined through
\begin{equation*}
%  \label{eq:def_Fk}
  \bF^\kappa(\bP):= \big (\kappa+\abs{\bP^\sym} \big )^{\frac
    {p-2}{2}}{\bP^\sym } \,.
\end{equation*}
Now we can proceed as in~\cite{bdr-7-5}. Indeed, from~\eqref{eq:fin} and the properties of
$\varphi_{p,\kappa}$ (in particular~\eqref{eq:equi3}) it follows that
$\bF^\kappa(\bD\bu_\kappa)$ is uniformly bounded in $W^{1,2}(\Omega)$, that $\bu_\kappa$
is uniformly bounded in $W^{1,p}_0(\Omega)$ and that $\bS^\kappa(\bD\bu_\kappa)$ is
uniformly bounded in $L^{p'}(\Omega)$. Thus, there exist $\bA \in W^{1,2}(\Omega)$,
$\bu\in W^{1,p}_0(\Omega)$, $\boldsymbol \chi \in L^{p'}(\Omega)$, and a subsequence
$\{\kappa_n\}$, with $\kappa_n\to0$, such that
  \begin{equation*}
    \begin{aligned}
      \bF(\bD\bu_{\kappa_n}) &\rightharpoonup \bA&&\text{in
      }W^{1,2}(\Omega)\,,
      \\
      \bF^{\kappa_n}(\bD\bu_{\kappa_n})&\to \bA\quad&&\text{in
      }L^{2}(\Omega) \text{ and a.e.~in }\Omega\,,
      \\
      \bu_{\kappa_n}&\rightharpoonup \bu&&\text{in }
      W^{1,p}_0(\Omega)\,,
      \\
      \bS^\kappa(\bD\bu_\kappa)&\rightharpoonup \boldsymbol \chi &&\text{in } L^{p'}(\Omega)\,.
    \end{aligned}
  \end{equation*}
Setting $\bB:=( \bF^0)^{-1}(\bA)$, it follows from
\cite[Lemma 3.23]{bdr-7-5}  that
\begin{equation*}
  \bD\bu_{\letter_n}=(\bF^{\letter_n})^{-1}(\bF^{\letter_n}(\bD\bu_{\letter_n})
  )\to( \bF^0)^{-1}(\bA)=\bB\quad\text{ a.e.~in } \Omega.  
\end{equation*}
Since weak and a.e.~limit coincide we obtain that
\begin{equation*}%\label{ae}
  \bD\bu_{\letter_n} \to \bD\bu=\bB \qquad \text{ a.e.~in }  \Omega\,. 
\end{equation*}
From~\cite[Lemma 3.16]{bdr-7-5} and~\cite[Corollary 3.22]{bdr-7-5} it now follows that
\begin{align*}
  \begin{aligned}
    \bF(\bD\bu_{\kappa_n}) &\rightharpoonup \bF^0(\bD\bu)&&\text{in
    }W^{1,2}(\Omega)\,,
    % \\
    % \bF^{\kappa_n}(\bD\bu_{\kappa_n})&\to \bF^0(\bD\bu)\quad
    % &&\text{in
    % }L^{2}(\Omega) \text{ and a.e.~in }\Omega\,,
    \\
    \bS^{\letter_n}(\bD\bu_{\letter_n}) &\to \bS(\bD\bu)
    &&\text{a.e.~in } \Omega\,.
  \end{aligned}
\end{align*}
Since weak and a.e.~limit coincide we obtain that
\begin{equation*}
%\label{ae}
   \bD\bu=\boldsymbol \chi \qquad \text{ a.e.~in }  \Omega\,. 
\end{equation*}
Now we can finish the proof in the same way as in the case
$\delta>0$. 
\end{proof}
\section*{Acknowledgments}
The research that led to the present paper was partially supported by a grant of the group
GNAMPA of INdAM.
% \bibliographystyle{%myams
% plain}
%\bibliography{lit,rose,fluid,lista}
\def\ocirc#1{\ifmmode\setbox0=\hbox{$#1$}\dimen0=\ht0 \advance\dimen0
  by1pt\rlap{\hbox to\wd0{\hss\raise\dimen0
  \hbox{\hskip.2em$\scriptscriptstyle\circ$}\hss}}#1\else {\accent"17 #1}\fi}
  \def\cprime{$'$} \def\cprime{$'$}
  \def\ocirc#1{\ifmmode\setbox0=\hbox{$#1$}\dimen0=\ht0 \advance\dimen0
  by1pt\rlap{\hbox to\wd0{\hss\raise\dimen0
  \hbox{\hskip.2em$\scriptscriptstyle\circ$}\hss}}#1\else {\accent"17 #1}\fi}
  \def\ocirc#1{\ifmmode\setbox0=\hbox{$#1$}\dimen0=\ht0 \advance\dimen0
  by1pt\rlap{\hbox to\wd0{\hss\raise\dimen0
  \hbox{\hskip.2em$\scriptscriptstyle\circ$}\hss}}#1\else {\accent"17 #1}\fi}
  \def\cprime{$'$} \def\cprime{$'$} \def\cprime{$'$} \def\cprime{$'$}
  \def\cprime{$'$} \def\ocirc#1{\ifmmode\setbox0=\hbox{$#1$}\dimen0=\ht0
  \advance\dimen0 by1pt\rlap{\hbox to\wd0{\hss\raise\dimen0
  \hbox{\hskip.2em$\scriptscriptstyle\circ$}\hss}}#1\else {\accent"17 #1}\fi}
  \def\cprime{$'$} \def\polhk#1{\setbox0=\hbox{#1}{\ooalign{\hidewidth
  \lower1.5ex\hbox{`}\hidewidth\crcr\unhbox0}}}

\end{document}